\documentclass[11pt,reqno]{amsart}
\usepackage{bbm}
\usepackage{amsmath, amssymb, amsfonts, amsthm, cases}
\usepackage{hyperref}
\usepackage[capitalise]{cleveref}
\usepackage{vmargin}
\setmarginsrb{2cm}{1cm}{2cm}{3cm}{1cm}{1cm}{2cm}{2cm}
\usepackage[dvipsnames]{xcolor}
\usepackage{mathrsfs}
\usepackage{graphicx}
\usepackage{upref}
\hypersetup{linkcolor=blue, colorlinks=true,citecolor = red}
\usepackage{amsaddr}

\newtheorem{thm}{Theorem}[section]
\newtheorem{prop}[thm]{Proposition}
\newtheorem{lem}[thm]{Lemma}

\newtheorem{defin}[thm]{Definition}

\theoremstyle{definition}

\theoremstyle{remark}
\numberwithin{equation}{section}

\newcommand{\mc}{\mathcal}

\renewcommand{\leq}{\leqslant}
\renewcommand{\geq}{\geqslant}

\DeclareMathOperator{\dist}{dist}
\renewcommand{\div}{\operatorname{div}}

\begin{document}

\title{Stabilization of a rigid body moving in a compressible viscous fluid}

\author{Arnab Roy and Tak\'eo Takahashi}
\address{Universit\'e de Lorraine, CNRS, Inria, IECL, F-54000 Nancy, France}

\date{\today}
 
\maketitle

\begin{abstract}
We consider the stabilizability of a fluid-structure interaction system where the fluid is viscous and compressible and the structure is a rigid ball.
The feedback control of the system acts on the ball and corresponds to a force that would be produced by a spring and a damper connecting the center of the ball to a fixed point $h_1$. We prove the global-in-time existence of strong solutions for the corresponding system under a smallness condition on the initial velocities and 
on the distance between the initial position of the center of the ball and $h_1$. 
Then, we show with our feedback law, that the fluid and the structure velocities go to 0 and that the center of the ball goes to $h_1$
as $t\to \infty$.
\end{abstract}

\bigskip


\tableofcontents

\section{Introduction and main result}
 Let $\Omega \subset \mathbb{R}^3$ be a bounded domain with $C^4$ boundary occupied by a fluid and a rigid body. We denote by $\mc{B}(t) \subset \Omega$, the domain of the rigid body and we assume it is an open ball of radius $1$ and of center $h(t)$, where $t \in \mathbb{R}_{+}$ is the time variable.
We suppose that the fluid domain $\mc{F}(t) = \Omega \setminus \overline{\mc{B}(t)}$ is connected.

The fluid is modeled by the compressible Navier-Stokes system whereas the motion of the rigid body is governed by the balance equations for linear and angular momentum. We also assume the no-slip boundary conditions. The equations of motion of fluid-structure are:
\begin{equation}\label{continuity}
\frac{\partial \rho}{\partial t} + \operatorname{div}(\rho u)=0 \quad t > 0, \, x\in\mc{F}(t),
\end{equation}
\begin{equation}\label{momentum}
\rho\left(\frac{\partial u}{\partial t}+ \left(u\cdot \nabla\right)u\right)- \div \sigma(u,p)=0 \quad  t > 0, \, x\in\mc{F}(t),
\end{equation}
\begin{align}
m\ell'&= -\int\limits_{\partial \mc{B}(t)} \sigma(u,p)N\, d\Gamma + w \quad t \geq 0,
\label{linear:body}
\\ 
J\omega'&=-\int\limits_{\partial \mc{B}(t)} (x-h(t)) \times \sigma(u,p)N\, d\Gamma\quad t \geq 0,
\label{angular:body} \\
h'&=\ell\quad t \geq 0,
\end{align}
\begin{align}
u(t,x)&=0 \quad  t>0, \ x \, \in \partial \Omega, \label{boundary f}\\
u(t,x)&= \ell(t)+\omega(t) \times (x-h(t))  \quad t>0, \  x \, \in \partial \mc{B}(t), \label{boundary fs}
\end{align}
\begin{align}
\rho(0,\cdot)=\rho_0, \quad u(0,\cdot)=u_0 \quad \mbox{in}\  \mc{F}(0),\\
h(0)=h_0,\quad \ell(0)=\ell_0,\quad \omega(0)=\omega_0. \label{initial}
\end{align}
In the above equations, $\rho=\rho(t,x)$ and $u=u(t,x)$ represent respectively the density and the velocity of the fluid and the pressure of the fluid is denoted by $p$. We assume that the flow is in the barotropic regime and we focus on the isentropic case where the relation between $p$ and $\rho$ is given by the constitutive law:
\begin{equation*}
p= a\rho^{\gamma},
\end{equation*} with $a>0$ and the adiabatic constant $\gamma > \frac{3}{2}$. The Cauchy stress tensor is defined as: 
\begin{equation*}
\sigma(u,p)=2\mu \mathbb{D}(u) + \lambda \operatorname{div}u\mathbb{I}_3 -p\mathbb{I}_3,
\end{equation*}
where $\mathbb{D}(u)=\frac{1}{2}\left(\nabla u + \nabla u^{\top}\right)$ denotes the symmetric part of the velocity gradient ($\nabla u^{\top}$ is the transpose of the matrix $\nabla u$) and $\lambda,\mu$ are the viscosity coefficients satisfying 
\begin{equation*}
\mu > 0, \quad \lambda + \mu \geq 0.
\end{equation*}
Here $\ell$ and $\omega$ are the linear and angular velocities of the rigid body, $N(t,x)$ is the unit normal to $\partial \mc{B}(t)$ at the point $x \in \partial \mc{B}(t)$, directed to the interior of the ball and $m$, $J$ are the mass and the moment of inertia of the rigid ball respectively. The formulae for $m$ and $J$ are
$$
m = \frac{4}{3}\pi\rho_{\mc{B}}, \quad J = \frac{2m}{5}\mathbb{I}_3,
$$
where $\rho_{\mc{B}} > 0$ is the constant density of the rigid ball.

Finally, $w$ (in \eqref{linear:body}) is our control that we take as a feedback control:
\begin{equation}\label{feedback}
w(t)=k_{p}(t)(h_1-h(t))-k_d\ell(t),
\end{equation}
where $k_d \geq 0$ and $k_p(t)\geq 0$ are well-chosen so that 
\begin{equation*}
\lim_{t\rightarrow \infty}h(t)= h_1,
\end{equation*}
whereas the velocities of the fluid and of the rigid ball go to $0$:
\begin{equation*}
\lim_{t\rightarrow \infty}u(t)= 0,\quad \lim_{t\rightarrow \infty}\ell(t)= 0,\quad \lim_{t\rightarrow \infty}\omega(t)= 0.
\end{equation*}
In literature, this type of control is known as Proportional-Derivative (PD) controller generated by a spring and a damper. 
The spring-damper is connected from the center of the ball to the fixed anchor point $h_1$ and it is attracting the ball towards the point $h_1$.

In order to give the precise statement of stabilization (\cref{asymptotic behavior}), we first need a global in time existence result for \eqref{continuity}--\eqref{feedback} with \eqref{feedback}. Such a result in the case without control is given in \cite{Muriel-Sergio-IHP} by adapting a method introduced in \cite{MR713680}.

 Here we will prove again this existence result, with the same approach but with a special attention to the estimates on $h(t)$ and with some modifications in the proof of \cite{Muriel-Sergio-IHP} due to the feedback law \eqref{feedback}.

In order to state our result we introduce $\overline{\rho}$ the mean-value of $\rho_{0}$:
\begin{equation}\label{mean value}
\overline{\rho}=\frac{1}{|\mc{F}(0)|}\int\limits_{\mc{F}(0)} \rho_{0}(x) \, dx.
\end{equation}
Note that, from equation \eqref{continuity} and Reynold's Transport Theorem, we obtain
\begin{equation*}
 \int\limits_{\mc{F}(0)} \rho_{0}(x) \, dx = \int\limits_{\mc{F}(t)} \rho(t,x) \, dx.
 \end{equation*}
For $0\leq T_1 < T_2 \leq \infty$, we introduce the following space: 
\begin{equation}
\begin{gathered}
{\widehat{\mc{S}}_{T_1,T_2}}=\Big\{(\rho,u,\ell,\omega) \mid 
\rho \in L^2 (T_1,T_2; H^3(\mc{F}(t)))\cap BC^0([T_1,T_2]; H^3(\mc{F}(t))) \cap H^1(T_1,T_2; H^2(\mc{F}(t))\\ 
\cap\ BC^1([T_1,T_2]; H^2(\mc{F}(t))) \cap H^2(T_1,T_2; L^2(\mc{F}(t))),\\
u \in L^2(T_1,T_2; H^4(\mc{F}(t))) \cap BC^0([T_1,T_2]; H^3(\mc{F}(t)))
\cap H^1(T_1,T_2; H^2(\mc{F}(t)) \\
\cap BC^1([T_1,T_2]; H^1(\mc{F}(t)))\cap\ H^2(T_1,T_2; L^2(\mc{F}(t))), \\
 \ell \in H^2(T_1,T_2), \quad \omega \in H^2(T_1,T_2)  
\Big\}.
\end{gathered}\label{solution space}
\end{equation}
Here $BC^k$ are the functions of class $C^k$ bounded with bounded derivatives.
We set 
\begin{equation}
\begin{gathered}
\|(\rho,u,\ell,\omega)\|_{\widehat{\mc{S}}_{T_1,T_2}} =\|\rho - \overline{\rho}\|_{L^{\infty}(T_1,T_{2}; H^3(\mc{F}(t)))}+ \|\rho - \overline{\rho}\|_{H^{1}(T_1,T_{2}; H^2(\mc{F}(t)))} + \|\rho - \overline{\rho}\|_{W^{1,\infty}(T_1,T_{2}; H^2(\mc{F}(t)))}\\ + \|\rho - \overline{\rho}\|_{H^2(T_1,T_{2}; L^2(\mc{F}(t)))} 
+ \|u\|_{L^2(T_1,T_2; H^4(\mc{F}(t)))} + \|u\|_{L^{\infty}(T_1,T_2; H^3(\mc{F}(t)))} 
+ \|u\|_{H^{1}(T_1,T_{2}; H^2(\mc{F}(t)))} \\
+ \|u\|_{W^{1,\infty}(T_1,T_2; H^1(\mc{F}(t)))} 
+ \|u\|_{H^2(T_1,T_2; L^2(\mc{F}(t)))} 
 + \|\ell\|_{H^2(T_1,T_{2})} + \|\ell\|_{W^{1,\infty}(T_1,T_{2})} \\
 + \|\omega\|_{H^2(T_1,T_{2})} + \|\omega\|_{W^{1,\infty}(T_1,T_{2})},
\end{gathered}\label{notation:norm} 
\end{equation}
and for $T>0$
\begin{equation*}
\|(\rho_0,u_0,\ell_0,\omega_0)\|_{\widehat{\mc{S}}_{T,T}} = \|\rho_0-\overline{\rho}\|_{H^3(\mc{F}(T))}+\|u_0\|_{H^3(\mc{F}(T))}+|\ell_0|+|\omega_0|.
\end{equation*}
Since we are working with regular solutions of \eqref{continuity}--\eqref{feedback}, we need to introduce the following compatibility conditions at initial time:
\begin{equation}\label{finalcompatibility-1}
u_0(y) = \ell_0 + \omega_0 \times (y-h_0)  \mbox{ for }y\in \partial\mc{B}(0),\quad u_{0} =0 \mbox{ on }\partial\Omega,
\end{equation}
\begin{equation}\label{finalcompatibility-2}
-\frac{1}{\rho_0} \div \sigma(u_0,p_0)=0  \mbox{ on }\partial\Omega,
 \end{equation}
 \begin{multline}\label{finalcompatibility-3}
-\Big(\omega_{0} \times (\omega_0 \times (y-h_0))\Big)
-\frac{1}{\rho_0} \div \sigma(u_0,p_0)(y)
\\
= \frac{1}{m}\left[ \int\limits_{\partial \mc{B}(0)} \sigma(u_0,p_0) n\, d\Gamma - k_d\ell_0\right]
+\left[J^{-1}\int\limits_{\partial \mc{B}(0)} (x-h_0)\times \sigma(u_0,p_0) n\, d\Gamma_x\right] \times (y-h_0)
\\
\quad \mbox{ for } y\in \partial\mc{B}(0),
\end{multline}
where 
$$
p_0=a\rho_0^{\gamma}.
$$
Finally, we introduce the following notation
$$
\Omega^0:=\left\{x\in \Omega \ ; \ \dist(x,\partial\Omega)>1 \right\}.
$$

Our hypotheses on $k_p$ and $k_d$ are the following ones:
\begin{equation}\label{hypkp}
k_p \in C^1(\mathbb{R}_{+},[0,1]), \ k_p(0)=0, \ k_p>0 \ \text{in} \ (0,\infty), \  k_p\equiv 1 \ \text{in} \ [T_I,\infty),\  0\leq k_{p}'< \frac{k_d}{2T_{I}^2}
\end{equation}
for some $T_I>0$.
\begin{thm}\label{global existence}
Assume that $\Omega^0$ is non empty and connected. Let $h_1 \in \Omega^0$ and $\overline{\rho} >0$. Assume $w$ is given by the feedback law \eqref{feedback} with $(k_p,k_d)$
satisfying \eqref{hypkp}. 
There exists $\delta >0$ such that for any
\begin{equation}\label{initial condition space:global}
h_0 \in \Omega^0,\, \rho_0 \in H^3(\mc{F}(0)),\ \rho_0 > 0, \ u_0 \in H^3(\mc{F}(0)),\ \ell_0,\ \omega_0 \in \mathbb{R}^3,
\end{equation}
 satisfying the compatibility conditions \eqref{finalcompatibility-1}--\eqref{finalcompatibility-3} with
\begin{equation}\label{smallness}
\|(\rho_0,u_0,\ell_0,\omega_0)\|_{\widehat{\mc{S}}_{0,0}}+|h_1-h_0|\leq \delta,
\end{equation}
the system \eqref{continuity}--\eqref{feedback} admits a unique strong solution
$(\rho,u,\ell,\omega) \in \widehat{\mc{S}}_{0,\infty}$, $h\in L^{\infty}(0,\infty)$.
Moreover, there exist $C,\eta >0$ such that
\begin{equation}\label{final estimate}
\|(\rho,u,\ell,\omega)\|_{\widehat{\mc{S}}_{0,\infty}} + \|\sqrt{k_p}(h_1-h)\|_{L^{\infty}(0,\infty)}
\leq 
C\left(\|(\rho_0,u_0,\ell_0,\omega_0)\|_{\widehat{\mc{S}}_{0,0}}+|h_1-h_0|\right),
\end{equation}
\begin{equation}\label{1451}
\dist(h(t),\partial\Omega)>1+\eta \quad (t\geq 0).
\end{equation}
\end{thm}
We are now in a position to state our stabilization result. 
\begin{thm}\label{asymptotic behavior}
With the notations and assumptions of \cref{global existence}, the solution $(\rho,u,h,\ell,\omega)$ of \eqref{continuity}-\eqref{feedback} satisfies
\begin{align}
\lim_{t \to \infty}\|\rho(t,\cdot)-\overline{\rho}\|_{H^2(\mc{F}(t))}=0,\quad \lim_{t \to \infty}\|u(t,\cdot)\|_{H^2(\mc{F}(t))}=0,\label{limit:fluid}\\
\lim_{t \to \infty} h(t)=h_1,\quad \lim_{t \to \infty}\ell(t)=0,\quad \lim_{t \to \infty}\omega(t)=0 \label{limit:solid}.
\end{align}
\end{thm}

During the last two decades, there has been a considerable interest in fluid-structure interaction problems involving moving interfaces. Broadly speaking, these types of models can be classified into two types: either the structure is moving inside the fluid or the structure is located at the boundary of the fluid domain. Since in this article we are interested in studying the motion of body inside the compressible fluid domain, below we mention related works from the literature concerning this case only.

The global-in-time existence (up to contact) of weak solutions for compressible viscous flow (for $\gamma \geq 2$) in a bounded domain of $\mathbb{R}^3$ interacting with a finite number of rigid bodies has been studied by Desjardins and Esteban \cite{MR1765138}. In \cite{Feireisl-ARMA}, Feireisl established the global existence result (for $\gamma > 3/2$) regardless of possible collisions of several rigid bodies or a contact of the rigid bodies with the exterior boundary. 
Regarding strong solutions, the existence and uniqueness of global solutions for small initial data have been achieved in \cite{Muriel-Sergio-IHP} in the Hilbert space framework by 
Boulakia and Guerrero as long as no collisions occur. Their work is based on a method proposed in \cite{MR713680} for a viscous compressible fluid (without structure).
In a $L^{p}$-$L^{q}$ setting, the authors in \cite{MR3356467} proved the existence and uniqueness of local-in-time strong solutions for the system composed by rigid bodies immersed into a viscous compressible fluid and in \cite{haak:hal-01619647}, the authors establish the global in time existence up to contact.

Let us mention some works related to the large time behavior of fluid-structure interaction system. 
In \cite{MR2001181}, the authors analyze the fluid-structure model in one space dimension where the fluid is governed by the viscous Burgers equation 
and the solid mass is moving by the difference of pressure at both sides of it. 
They obtain that the asymptotic profile of the fluid is a self-similar solution of the Burgers equation and the point mass enjoys the parabolic trajectory as $t\rightarrow \infty$. 
An extension of this work in several space dimensions is obtained in \cite{MR2131059} for the heat equations in interaction with a rigid body.
Their result is that as $t \rightarrow \infty$, the fluid solution behaves as the fundamental solution of the heat equation and the ball goes to infinity in bidimensional case 
whereas the ball remains in a bounded domain in three dimension. 
Regarding the long-time behavior of a moving particle inside a Navier-Stokes fluid, the authors in \cite{MR2763535} consider in particular 
the case of a ball falling over an horizontal plane and show that the velocity of the fluid goes to zero and the particle reaches the bottom of the container asymptotically in time.
In \cite{MR3207006}, the authors analyze the case of a rigid disk immersed into a two-dimensional Navier-Stokes equations filling the exterior of the structure domain.
They restrict to the case of a solid and a fluid with the same density and for the linear case.

Finally, let us mention two works using a control supported on the rigid body:  \cite{Cindea} in the $1$d case for a Burgers-particle system and \cite{TT-MT-GW} in the $3$d case for a rigid ball moving into a viscous incompressible fluid. The main difference between this study and the two previous references come from the fact that 
in our case we need to deal with stronger solutions than in the incompressible case. In particular, to avoid compatibility conditions at $t=0$ that involve the feedback control $w$, we take here $k_p$ depending on time with $k_p(0)=0$.

The plan of the paper is the following. In Section \ref{sec:Local in time existence of solution}, we establish the local-in-time existence of solutions for the system \eqref{continuity}--\eqref{feedback}. We then obtain a priori estimates in Section \ref{sec:Global in time existence of solution} to prove \cref{global existence}. Finally Section \ref{sec:Proof of Theorem asymptotic behavior} is devoted to the asymptotic analysis of the solutions in order to prove \cref{asymptotic behavior}.

\subsection*{Notation}
For any $a\in \mathbb{R}^3$, we set 
\begin{equation*}
\widehat{\mc{B}}(a)=\{x\in \mathbb{R}^3\mid |x-a|< 1\}, \quad \widehat{\mc{F}}(a) = \Omega \setminus \overline{\widehat{\mc{B}}(a)}.
\end{equation*}
In particular, 
\begin{equation*}
\mc{B}(t)= \widehat{\mc{B}}(h(t)),\quad \mc{F}(t)=\widehat{\mc{F}}(h(t)).
\end{equation*}
In this article, to shorten the notation, we write $H^m$ and $L^2$ instead of $H^m(\mc{F}(0))$ and $L^2(\mc{F}(0))$.

Assume $\mathfrak{X}$ is Banach space. We need to consider a particular norm for $H^m(0,T;\mathfrak{X})$ if $m\in \mathbb{N}^*$ and if $T\in \mathbb{R}^*_+$. 
\begin{equation}\label{HmX}
\|f\|_{H^m_{\infty}(0,T;\mathfrak{X})} = \|f\|_{H^m(0,T;\mathfrak{X})} + \|f\|_{W^{m-1,\infty}(0,T;\mathfrak{X})}.
\end{equation}
Using the Sobolev embedding, this norm is equivalent to the usual one, but the corresponding constants depend on $T$ and that is the reason why we introduce such a notation.

Assume $\mathfrak{X}_1$ and $\mathfrak{X}_2$ are Banach spaces. We also introduce the following spaces
$$
H^m(0,T;\mathfrak{X}_1,\mathfrak{X}_2) = L^2(0,T;\mathfrak{X}_1) \cap H^m(0,T;\mathfrak{X}_2) \quad (m\geq 1).
$$
In the case $T\in \mathbb{R}^*_+$, we also need to introduce the following norm for the above space:
\begin{equation}\label{infinity-1}
\|f\|_{H^1_{\infty}(0,T; H^2, L^2)} = \|f\|_{L^2(0,T; H^2)} + \|f\|_{L^{\infty}(0,T; H^1)} + \|f\|_{H^1(0,T; L^2)},
\end{equation}
\begin{equation}\label{infinity-2}
\|f\|_{H^2_{\infty}(0,T; H^4, L^2)} = \|f\|_{L^2(0,T; H^4)} + \|f\|_{L^{\infty}(0,T; H^3)} + \|f\|_{H^1(0,T; H^2)} + \|f\|_{W^{1,\infty}(0,T; H^1)} + \|f\|_{H^2(0,T; L^2)}.
\end{equation}
Using interpolation results, we see again that the corresponding norm is equivalent to $H^1(0,T; H^2)$ but the corresponding constants depend on $T$.

\section{Local in time existence of solutions}\label{sec:Local in time existence of solution}
In order to prove Theorem \ref{global existence}, we first prove the existence and uniqueness of strong solutions of system \eqref{continuity}-\eqref{feedback}
for small times. More precisely, we show in this section the following result:
\begin{thm}\label{local-in-time existence}
Let $h_1 \in \Omega^0$ and $\overline{\rho} >0$. Assume $w$ is given by the feedback law \eqref{feedback} with $k_d \in \mathbb{R}$ and $k_p \in H^1_{loc}([0,\infty))$. There exist $\delta_0,C_{*},T_* >0$ such that for any
\begin{equation}\label{initial cond req-1}
h_0 \in \Omega^0,\ \rho_0 \in H^3,\ u_0 \in H^3,\ \ell_0,\ \omega_0 \in \mathbb{R}^3,
\end{equation}
 satisfying the compatibility conditions \eqref{finalcompatibility-1}--\eqref{finalcompatibility-3} with
\begin{equation}\label{initial cond req-2}
 \|(\rho_0,u_0,\ell_0,\omega_0)\|_{\widehat{\mc{S}}_{0,0}}+|h_1-h_0|\leq \delta_0,
 \end{equation}
the system \eqref{continuity}-\eqref{initial} admits a unique strong solution
$(\rho, u, \ell, \omega) \in \widehat{\mc{S}}_{0,T_*}$, $h\in L^{\infty}(0,T_*)$ 
and
\begin{equation}\label{final local estimate}
\|(\rho,u,\ell,\omega)\|_{\widehat{\mc{S}}_{0,T_*}}+ \|h_1-h\|_{L^{\infty}(0,T_*)} \leq 
C_*\Big(\|(\rho_0,u_0,\ell_0,\omega_0)\|_{\widehat{\mc{S}}_{0,0}}+|h_1-h_0|\Big).
\end{equation} 
\end{thm}

\subsection{Lagrangian change of variables}\label{Change of variables}
Firstly, we use a Lagrangian change of variables to rewrite the system \eqref{continuity}--\eqref{feedback} in a fixed spatial domain: let introduce the flow $X(t,\cdot): \overline{\mc{F}(0)} \rightarrow \overline{\mc{F}(t)}$
defined by 
$$
\begin{cases}
\displaystyle\frac{\partial X}{\partial t}(t,y)=u(t,X(t,y)), \\ X(0,y)=y.
\end{cases}
$$
Due to the boundary conditions, we have
$$
X(t,y)=
\begin{cases}
h(t) + Q(t)(y-h_0)&\text{if} \ y\in \partial\mc{B}(0), \\ 
y&\text{if} \ y\in \partial\Omega,
\end{cases}
$$
where $Q(t) \in SO(3)$ is the rotation matrix associated to the angular velocity $\omega$:
$$
Q' =\mathbb{A}(\omega)Q,\quad Q(0) =\mathbb{I}_3.
$$
For any $\omega\in \mathbb{R}^3$, $\mathbb{A}(\omega)$ is the skew-symmetric matrix:
$$\mathbb{A}(\omega)=\begin{pmatrix}
0 & -\omega_3 & \omega_2 \\
\omega_3 & 0 & -\omega_1 \\
-\omega_2 & \omega_1 & 0
\end{pmatrix}.
$$
If $u$ is regular enough, $X$ is well-defined and $X(t,\cdot)$ is a $C^{1}$-diffeomorphism from $\overline{\mc{F}(0)}$ onto $\overline{\mc{F}(t)}$ for all $t \in (0,T)$. 
We denote by $Y(t,\cdot)$ the inverse of $X(t,\cdot)$ and we consider the following change of variables 
\begin{align}
\widetilde{u}(t,y)=Q(t)^{\top}u(t,X(t,y)),\quad \widetilde{\rho}(t,y)=\rho(t,X(t,y))-\overline{\rho}, \label{chng of var:fluid}\\
\widetilde{h}(t)=h(t)-h_1,\quad \widetilde{\ell}(t)=Q(t)^{\top}\ell(t), \quad \widetilde{\omega}(t)=Q(t)^{\top}\omega(t) \label{chng of var:body}. 
\end{align}
Note that now we have
\begin{equation}\label{def:X}
X(t,y)= y + \int\limits_{0}^{t} Q(s)\widetilde{u}(s,y) ds, \quad \forall \, y\in \overline{\mc{F}(0)}.
\end{equation}

Under the change of variables \eqref{chng of var:fluid}-\eqref{chng of var:body}, the system \eqref{continuity}-\eqref{initial} is transformed as follows:
\begin{align}
\frac{\partial \widetilde{\rho}}{\partial t} + \rho_{0}\operatorname{div}\widetilde{u} &= F_{1}(\widetilde{\rho},\widetilde{u},\widetilde{\ell},\widetilde{\omega},Q)   \quad \mbox{ in }(0,T)\times \mc{F}(0), \label{reform:fluiddensity}\\
\frac{\partial \widetilde{u}}{\partial t} - \frac{\mu}{\rho_{0}}\Delta \widetilde{u} - \frac{\lambda+\mu}{\rho_{0}}\nabla \left(\operatorname{div} \widetilde{u}\right)&=F_{2}(\widetilde{\rho},\widetilde{u},\widetilde{\ell},\widetilde{\omega},Q)  \quad\mbox{ in }\quad (0,T)\times \mc{F}(0), \label{reform:fluidvelocity} \\ 
m\widetilde{\ell}'  &=F_{3}(\widetilde{\rho},\widetilde{u},\widetilde{h},\widetilde{\ell},\widetilde{\omega},Q) \quad\mbox{ in }\quad (0,T),  \label{reform:rigidlinear}\\
J\widetilde{\omega}' &=F_{4}(\widetilde{\rho},\widetilde{u},\widetilde{\ell},\widetilde{\omega},Q) \quad\mbox{ in }\quad (0,T),\label{reform:rigidangular}\\
\widetilde{h}' &= Q\widetilde{\ell},\, Q' =Q\mathbb{A}(\widetilde{\omega}) \quad\mbox{ in }\quad (0,T), \label{reform:position}\\
\widetilde{u}&=\widetilde{\ell}+\widetilde{\omega}\times (y-h_0) \quad\mbox{ on }\quad (0,T)\times \partial \mc{B}(0), \label{reformboundary1:vel}\\
\widetilde{u}&=0 \quad\mbox{ in }\quad (0,T)\times \partial \Omega,\label{reformboundary2:vel}\\
\widetilde{\rho}(0,\cdot)&=\rho_0(\cdot) - \overline{\rho},\quad \widetilde{u}(0,\cdot)=u_0(\cdot), \quad\mbox{ in }\quad \mc{F}(0), \label{reforminitialcond:densityvel}\\
\widetilde{h}(0)&=h_0-h_1,\quad \widetilde{\ell}(0) =\ell_{0}, \quad \widetilde{\omega}(0) =\omega_0, \quad Q(0) =\mathbb{I}_3.\label{reformbody:initial}
\end{align}

In the above equations, $F_1, F_2, F_3, F_{4}$ are defined in the following way:
\begin{equation}\label{F1}
F_{1}(\widetilde{\rho},\widetilde{u},\widetilde{\ell},\widetilde{\omega},Q) = - (\widetilde{\rho}+\overline{\rho})\nabla \widetilde{u}: \left[((\nabla Y(X))Q)^{\top} - \mathbb{I}_3\right] - (\widetilde{\rho} + \overline{\rho} - \rho_{0})\operatorname{div}\widetilde{u},
\end{equation}
for $i=1,2,3$:
\begin{multline}\label{F2}
(F_{2})_{i}(\widetilde{\rho},\widetilde{u},\widetilde{\ell},\widetilde{\omega},Q)= - (\widetilde{\omega}\times \widetilde{u})_{i} + \frac{\mu}{\widetilde{\rho}+\overline{\rho}}\sum\limits_{p,l,m}\frac{\partial^{2}\widetilde{u}_{i}}{\partial y_{m}\partial y_{l}}\left(\frac{\partial Y_{m}}{\partial x_{p}}(X)\frac{\partial Y_{l}}{\partial x_{p}}(X) - \delta_{mp}\delta_{lp}\right) \\+ \frac{\mu}{\widetilde{\rho}+\overline{\rho}}\sum\limits_{p,l} \frac{\partial \widetilde{u}_{i}}{\partial y_{l}}\frac{\partial^2 Y_{l}}{\partial x_{p}^2}(X)  +\mu \Delta\widetilde{u}_{i}\left(\frac{\rho_{0}-(\widetilde{\rho}+\overline{\rho})}{\rho_{0}(\widetilde{\rho}+\overline{\rho})}\right) + \frac{\lambda + \mu}{\widetilde{\rho}+\overline{\rho}}\sum\limits_{p,l} \frac{\partial \widetilde{u}_{p}}{\partial y_{l}}\frac{\partial^2 Y_{l}}{\partial x_{p}\partial x_{i}}(X) \\ + \frac{\lambda+\mu}{\widetilde{\rho}+\overline{\rho}}\sum\limits_{p,l,m}\frac{\partial^{2}\widetilde{u}_{p}}{\partial y_{m}\partial y_{l}}\left(\frac{\partial Y_{m}}{\partial x_{p}}(X) - \delta_{mp}\right)\frac{\partial Y_{l}}{\partial x_{i}}(X) + \frac{\lambda + \mu}{\widetilde{\rho}+\overline{\rho}}\sum\limits_{p,l}\frac{\partial^{2}\widetilde{u}_{p}}{\partial y_{p}\partial y_{l}}\left(\frac{\partial Y_{l}}{\partial x_{i}}(X) - \delta_{li}\right) \\+(\lambda + \mu) [\nabla (\operatorname{div}\widetilde{u}) ]_{i}\left(\frac{\rho_{0}-(\widetilde{\rho}+\overline{\rho})}{\rho_{0}(\widetilde{\rho}+\overline{\rho})}\right) + a\gamma(\widetilde{\rho}+\overline{\rho})^{\gamma - 2}\sum\limits_{j,l}Q_{ji}\frac{\partial \widetilde{\rho}}{\partial y_{l}}\frac{\partial Y_{l}}{\partial x_{j}}(X) ,
\end{multline}
\begin{multline}\label{F3}
F_3(\widetilde{\rho},\widetilde{u}, \widetilde{h},\widetilde{\ell},\widetilde{\omega},Q)=-m(\widetilde{\omega}\times \widetilde{\ell}) -\int\limits_{\partial \mc{B}(0)} \Big[\mu\left(Q\nabla\widetilde{u}(\nabla Y(X))+ (Q\nabla\widetilde{u}(\nabla Y(X)))^{\top}\right)\\ + \lambda \left(Q\nabla \widetilde{u}(\nabla Y(X)) : \mathbb{I}_{3} \right)-a(\overline{\rho}+\widetilde{\rho})^{\gamma}\Big]n\, d\Gamma - k_pQ^{\top}\widetilde{h} - k_d\widetilde{\ell},
\end{multline}
\begin{multline}\label{F4}
F_4(\widetilde{\rho},\widetilde{u},\widetilde{\ell},\widetilde{\omega},Q)= -\int\limits_{\partial \mc{B}(0)} (y-h_0) \times \Big[\mu\left(Q\nabla\widetilde{u}(\nabla Y(X))+ (Q\nabla\widetilde{u}(\nabla Y(X)))^{\top}\right) \\+ \lambda \left(Q\nabla \widetilde{u}(\nabla Y(X)) : \operatorname{Id} \right)-a(\overline{\rho}+\widetilde{\rho})^{\gamma}\Big]n\, d\Gamma.
\end{multline}
Here $n(y)=Q(t)^{\top}N(t,x)$ is the unit normal to $\partial \mc{B}(0)$ at the point $y \in \partial \mc{B}(0)$, directed to the interior of the ball.

\subsection{Analysis of a linear problem}
In this section, we want to study the existence and regularity of the solution of the following linear system:
\begin{align}
\frac{\partial \widetilde{\rho}}{\partial t} + \rho_{0}\operatorname{div}\widetilde{u} &= f_{1}\quad\mbox{ in }\quad (0,T)\times \mc{F}(0), \label{linearfluid:density}\\
\frac{\partial \widetilde{u}}{\partial t} - \frac{\mu}{\rho_{0}}\Delta \widetilde{u} - \frac{\lambda+\mu}{\rho_{0}}\nabla \left(\operatorname{div} \widetilde{u}\right)&=f_{2}\quad\mbox{ in }\quad (0,T)\times \mc{F}(0), \label{linearfluid:vel}\\
m\widetilde{\ell}'  &=f_3 \quad\mbox{ in }\quad (0,T),\label{body:linearmom}\\
J\widetilde{\omega}'  &= f_4 \quad\mbox{ in }\quad (0,T),\label{body:angularmom}\\
\widetilde{u}&=\widetilde{\ell}+\widetilde{\omega}\times (y-h_0)\quad\mbox{ on }\quad (0,T)\times \partial \mc{B}(0), \label{boundary1:vel}\\
\widetilde{u}&=0 \quad\mbox{ on }\quad (0,T)\times \partial \Omega, \label{boundary2:vel}\\
\widetilde{u}(0,\cdot)&=u_0(\cdot) \quad\mbox{ in }\quad \mc{F}(0), \label{initialcond:vel}\\
\widetilde{\rho}(0,\cdot)&=\widetilde{\rho}_0 \quad\mbox{ in }\quad \mc{F}(0), \label{initialcond:density}\\
\widetilde{\ell}(0) &=\ell_{0}, \quad \widetilde{\omega}(0) =\omega_0. \label{linearbody:initial}
\end{align}
We introduce the following set for $T>0$:
\begin{multline}\label{13:40}
\mc{S}_{T} = \Big\{ (\widetilde{\rho},\widetilde{u},\widetilde{\ell},\widetilde{\omega}) \mid \widetilde{\rho} \in  H^1(0,T; H^3) \cap C^1([0,T]; H^2) \cap H^2(0,T; L^2), \widetilde{u}\in H^2(0,T; H^4, L^2),\\
\widetilde{\ell} \in H^2(0,T),\, \widetilde{\omega} \in H^2(0,T), \ \widetilde{u}=0 \mbox{ on }\partial\Omega,
\ \widetilde{u}= \widetilde{\ell} + \widetilde{\omega}\times (y-h_0) \mbox{ on }\partial\mc{B}(0),
\ \widetilde{\rho}(0)=\widetilde{\rho}_0,\\ \widetilde{u}(0)=u_0, \, 
\widetilde{\ell}(0) =\ell_{0},\, \widetilde{\omega}(0) =\omega_0\Big\},
\end{multline} 
\mbox{equipped with the norm} 
\begin{multline*} 
\|(\widetilde{\rho},\widetilde{u},\widetilde{\ell},\widetilde{\omega})\|_{\mc{S}_{T}} :=  \|\widetilde{\rho}\|_{H^1_{\infty}(0,T; H^3)} + \|\widetilde{\rho}\|_{W^{1,\infty}(0,T; H^2)} + \|\widetilde{\rho}\|_{H^2(0,T; L^2)} + \|\widetilde{u}\|_{H^2_{\infty}(0,T; H^4,L^2)}\\+ \|\widetilde{\ell}\|_{H^2_{\infty}(0,T)}+ \|\widetilde{\omega}\|_{H^2_{\infty}(0,T)}. 
\end{multline*}
We recall that the norms $\|\cdot\|_{H^1_{\infty}(0,T; H^3)},$ $\|\cdot\|_{H^2_{\infty}(0,T)}$ are defined in \eqref{HmX} and $\|\cdot\|_{H^2_{\infty}(0,T; H^4,L^2)}$ is defined in \eqref{infinity-2}. The space $\mc{S}_T$ is similar to $\widehat{\mc{S}}_{T_1,T_2}$ defined by \eqref{solution space} except that here $\mc{F}(t)$ is replaced by $\mc{F}(0)$ and we add the boundary and initial conditions.

Since $\overline{\rho}>0$, there exists $\delta_0>0$ such that \eqref{initial cond req-2} implies 
$$
\rho_0\geq \frac{\overline{\rho}}{2}>0.
$$
In that case, the system \eqref{linearfluid:density}--\eqref{linearbody:initial} is well-posed:
\begin{prop}\label{linear system:existence}
Let us assume $\overline{\rho}>0$, \eqref{initial cond req-2} with $\delta_0$ as above and
\begin{gather*}
(\widetilde \rho_0,u_0,\ell_0,\omega_0)\in H^3 \times H^3 \times \mathbb{R}^3 \times \mathbb{R}^3,\quad f_1 \in L^2(0,T; H^3)\cap C([0,T]; H^2) \cap H^1(0,T; L^2),\\ f_2 \in H^1(0,T; H^2, L^2),\quad f_3 \in H^1(0,T),\quad f_{4} \in H^1(0,T)
\end{gather*}
with
\begin{equation}\label{comp cond1}
u_0 = \ell_0 + \omega_0 \times (y-h_0) \mbox{ for }y\in \partial\mc{B}(0),\quad u_{0} =0 \mbox{ on }\partial\Omega,
\end{equation}
\begin{equation}\label{comp cond2}
 f_2(0) + \frac{\mu}{\rho_{0}}\Delta u_{0} + \frac{\lambda+\mu}{\rho_{0}}\nabla \left(\operatorname{div} u_{0}\right)=0 \mbox{ on }\partial\Omega,
\end{equation}
\begin{equation}\label{comp cond3}
f_2(0) + \frac{\mu}{\rho_{0}}\Delta u_{0} + \frac{\lambda+\mu}{\rho_{0}}\nabla \left(\operatorname{div} u_{0}\right)=m^{-1}f_3(0) + J^{-1}f_4(0) \times (y-h_0) \mbox{ for }y\in \partial\mc{B}(0).
\end{equation}
Then the  system \eqref{linearfluid:density}--\eqref{linearbody:initial} admits a unique solution $(\widetilde{\rho},\widetilde{u},\widetilde{\ell},\widetilde{\omega})\in \mc{S}_T$. 
 Moreover, there exists $C_{L}>0$ (nondecreasing with respect to $T$) such that 
\begin{multline}\label{est:linearsystemfull}
\|(\widetilde{\rho},\widetilde{u},\widetilde{\ell},\widetilde{\omega})\|_{\mc{S}_{T}} \leq C_{L} \Big(\|f_1\|_{L^2(0,T; H^3)} +\|f_1\|_{L^{\infty}(0,T; H^2)} + \|f_1\|_{H^1(0,T; L^2)} + \|f_{2}\|_{H^1_{\infty}(0,T; L^2,H^2)} \\+ \|f_3\|_{H^1_{\infty}(0,T)} +\|f_{4}\|_{H^1_{\infty}(0,T)} + \|\widetilde \rho_0\|_{H^3} + \|u_0\|_{H^3} + |\ell_0| + |\omega_0| \Big).
\end{multline}
\end{prop}
\begin{proof}
We solve \eqref{linearfluid:density}-\eqref{linearbody:initial} like a cascade system: first, 
\eqref{body:linearmom}-\eqref{body:angularmom} admits a unique solution $(\widetilde{\ell},\widetilde{\omega})$
with
\begin{equation}\label{est:body}
	 \|\widetilde{\ell}\|_{H^2_{\infty}(0,T)}
		+ \|\widetilde{\omega}\|_{H^2_{\infty}(0,T)} 
\leq C \Big( \|f_3\|_{H^1_{\infty}(0,T)} + \|f_{4}\|_{H^1_{\infty}(0,T)}
	 + |\ell_0| + |\omega_0| \Big).
\end{equation}

Next, we solve equation \eqref{linearfluid:vel} with the boundary and initial conditions \eqref{boundary1:vel}-\eqref{initialcond:vel}.
First we consider a lifting operator $\mc{R}$, such that for any $a,b\in \mathbb{R}^3$, $\mc{R}(a,b)\in C^\infty(\mathbb{R}^3)$ satisfies
\begin{equation*}
\mc{R}(a,b)=
\begin{cases}
a+b\times (y-h_0)&\mbox{ on } \partial \mc{B}(0),\\
0 & \mbox{ on }  \partial \Omega.
\end{cases}
\end{equation*}
Then $\widetilde{v}=\widetilde{u}-\mc{R}(\widetilde{\ell},\widetilde{\omega})$ satisfies
\begin{equation*}
\left\{
        \begin{array}{ll}
        \displaystyle
\frac{\partial \widetilde{v}}{\partial t} - \dfrac{\mu}{\rho_{0}}\Delta \widetilde{v} - \dfrac{\lambda+\mu}{\rho_{0}}\nabla \left(\operatorname{div} \widetilde{v}\right)&=F=f_{2} + \dfrac{\mu}{\rho_{0}}\Delta \mc{R}(\widetilde{\ell},\widetilde{\omega}) + \dfrac{\lambda+\mu}{\rho_{0}}\nabla \left(\operatorname{div} \mc{R}(\widetilde{\ell},\widetilde{\omega})\right)-\mc{R}(\widetilde{\ell}',\widetilde{\omega}'),\\
&\widetilde{v}=0\quad\mbox{ on }\quad (0,T)\times \partial \mc{F}(0), \\ 
&\widetilde{v}(0,\cdot)=\widetilde{v}_0=u_0-\mc{R}({\ell}_0,{\omega}_0) \quad\mbox{ in }\quad \mc{F}(0).
\end{array}
        \right.
\end{equation*}
By using a standard Galerkin method (see \cite[Chapter 7, Theorem 1, p.354]{MR1625845}) and by using the regularity result of Lam\'e operator (see, for instance, \cite[Theorem 6.3-6, p.296]{MR936420}), 
under the condition that $\partial \mc{F}(0)$ is of class $C^4$, we can show the following result:
if
$$
F\in  L^2(0,T; H^2)\cap H^1(0,T; L^2),\quad \widetilde{v}_0\in H^3 \cap H^1_0,
$$
with the condition
\begin{equation}\label{13:38}
F(0,\cdot)+ \dfrac{\mu}{\rho_{0}}\Delta \widetilde{v}_0 + \dfrac{\lambda+\mu}{\rho_{0}}\nabla \left(\operatorname{div} \widetilde{v}_0\right)=0 \quad \text{on}\ \partial \mc{F}(0),
\end{equation}
then there exists a unique solution $\widetilde{v} \in H^2(0,T; H^4,L^2)$ with the estimate
$$
\|\widetilde{v}\|_{H^2_{\infty}(0,T; H^4,L^2)} \leq C\Big( \|F\|_{H^1_{\infty}(0,T; L^2,H^2)} + \|\widetilde{v}(0)\|_{H^3} \Big).
$$
We note that condition \eqref{13:38} is equivalent to \eqref{comp cond2} and \eqref{comp cond3}.
We can use the relation $\widetilde{u}=\widetilde{v}+\mc{R}(\widetilde{\ell},\widetilde{\omega})$ and the above estimate of $\widetilde{v}$ to deduce the following estimate of $\widetilde{u}$:
\begin{multline}\label{reg:u}
\|\widetilde{u}\|_{H^2_{\infty}(0,T; H^4,L^2)} \leq C\Big(\|f_2\|_{H^1_{\infty}(0,T; L^2,H^2)} + \|f_3\|_{H^1_{\infty}(0,T)} +\|f_{4}\|_{H^1_{\infty}(0,T)}    + \|u_0\|_{H^3} + |\ell_0| + |\omega_0| \Big).
\end{multline}
Now, with the help of equation \eqref{linearfluid:density} satisfied by $\widetilde{\rho}$, we obtain
\begin{multline}\label{est:fluiddensity}
\|\widetilde{\rho}\|_{H^1_{\infty}(0,T; H^3)} + \|\widetilde{\rho}\|_{W^{1,\infty}(0,T; H^2)} + \|\widetilde{\rho}\|_{H^2(0,T; L^2)}  \leq C\Big(\|f_1\|_{L^2(0,T; H^3)} +\|f_1\|_{L^{\infty}(0,T; H^2)} + \|f_1\|_{H^1(0,T; L^2)} \\+ \|\widetilde{u}\|_{L^2(0,T; H^4)} +\|\widetilde{u}\|_{L^{\infty}(0,T; H^3)} + \|\widetilde{u}\|_{H^1(0,T; H^1)} 
+ \|\widetilde\rho_0\|_{H^3}\Big). 
\end{multline}
Thus, we have proved the existence of solution in appropriate space for the system \eqref{linearfluid:density}-\eqref{linearbody:initial}. Thanks to \eqref{est:body}, \eqref{reg:u} and \eqref{est:fluiddensity}, we have also obtained our required estimate \eqref{est:linearsystemfull}. 
\end{proof} 

\subsection{Estimates of the nonlinear terms}\label{15:29}
For $T>0$ and $R>0$, we define the following subset of $\mc{S}_{T}$:
\begin{equation}\label{ball for fixed point}
\mc{S}_{T,R} = \left\{ (\widetilde{\rho},\widetilde{u},\widetilde{\ell},\widetilde{\omega}) \in \mc{S}_{T} \mid \|  (\widetilde{\rho},\widetilde{u},\widetilde{\ell},\widetilde{\omega}) \|_{\mc{S}_{T}}  \leqslant R \right\}.
\end{equation}
In what follows, $R$ is fixed and the constants that appear can depend on $R$.

Assume $(\widetilde{\rho},\widetilde{u},\widetilde{\ell},\widetilde{\omega}) \in \mc{S}_{T,R}$. Then
there exists a unique solution $(\widetilde{h},Q)\in H^3(0,T)$ of the following equations
\begin{equation}\label{13:43}
\begin{cases}
\widetilde{h}'= Q\widetilde{\ell}\quad\mbox{ in }\quad (0,T),\\
Q' =Q\mathbb{A}(\widetilde{\omega}) \quad\mbox{ in }\quad (0,T),
\\
Q(0)=\mathbb{I}_3,\quad \widetilde{h}(0)= h_0-h_1,\\
\end{cases}
\end{equation}
and we can then define $X$ by \eqref{def:X}.
From \eqref{ball for fixed point}, there exists $C=C(R)>0$ such that
$$
\|Q\|_{H^3(0,T)} \leq C,\quad \|Q-\mathbb{I}_3\|_{L^\infty(0,T)}\leq CT,
$$
\begin{equation}\label{est:h}
\|\widetilde{h}\|_{L^{\infty}(0,T)} \leq |h_0-h_1| + CT^{1/2}.
\end{equation}
In particular, taking $\delta_0$ small enough in \eqref{initial cond req-2}, there exists $T_1=T_1(R, \delta_0, \dist(h_1,\partial \Omega))>0$ 
and $c_1>0$ such that
\begin{equation}\label{12:07}
\dist(\widehat{\mathcal{B}}(\widetilde{h}(t)+h_1),\partial \Omega)\geq c_1>0 \quad \forall t\in [0,T_1].
\end{equation}
From now on, we assume $T\leq T_1$ and the constants may depend on $T_1$.

Combining \eqref{def:X} and \eqref{ball for fixed point}, we also deduce 
\begin{equation}\label{inverse:X}
\|\nabla X - \mathbb{I}_3\|_{L^{\infty}(0,T; H^3)} \leqslant  CT^{1/2}.
\end{equation}
In particular, using the embedding $H^3(\mathcal{F}(0)) \hookrightarrow W^{1,\infty}(\mathcal{F}(0))$ and \eqref{12:07}, there exists $T_2\leq T_1$ such that  
$X : \mc{F}(0) \to \widehat{\mc{F}}(\widetilde{h}(t)+h_1)$ is invertible and its inverse is denoted by $Y$. 

In the same spirit, using the initial condition on $\widetilde{\rho}$ (see \eqref{13:40}), we have
\begin{equation}\label{15:06bis}
\| \widetilde{\rho}+\overline{\rho} - \rho_0\|_{L^{\infty}(0,T;H^3)} \leq T^{1/2}R.
\end{equation}
Using the embedding $H^3(\mathcal{F}(0)) \hookrightarrow L^{\infty}(\mathcal{F}(0))$ and \eqref{initial cond req-2} 
with $\delta_0$ small enough, there exists $T_3\leq T_2$ such that  
\begin{equation}\label{15:06}
\frac {\overline{\rho}}2 \leq \widetilde{\rho}+\overline{\rho} \leq \frac {3\overline{\rho}}2.
\end{equation}
In particular, combining this with \eqref{ball for fixed point}, for any $\alpha\in \mathbb{R}$, 
\begin{equation}\label{rhoalpha}
\|(\widetilde{\rho}+\overline{\rho})^\alpha\|_{L^{\infty}(0,T;H^{3})} \leq C,
\quad
\left\|\int_{\partial \mc{B}(0)} (\widetilde{\rho}+\overline{\rho})^{\gamma} n d\Gamma\right\|_{H^1(0,T)} \leq CT^{1/2}.
\end{equation}

From the above construction and assuming $T\leq T_3$, we can define the terms $F_1, F_2, F_3, F_{4}$ by \eqref{F1}-\eqref{F4}.
To estimate these terms, we first give some estimates of $X$ and $Y$: 
\begin{lem} \label{lem:estX} 
Assume $(\widetilde{\rho},\widetilde{u},\widetilde{\ell},\widetilde{\omega}) \in \mc{S}_{T,R}$.
There exists a positive constant $C$ depending only on $R$, $\mc{F}(0)$ such that, for all $0 < T \leqslant T_{3}$,
\begin{gather} 
\|\nabla Y(X) - \mathbb{I}_3\|_{L^{\infty}(0,T; H^3)} \leq CT^{1/2},\label{X-6}\\
 \left\|\frac{\partial^2 Y_{l}}{\partial x_p\partial x_{i}}(X)\right\|_{L^{\infty}(0,T; H^2)} 
 +
 \left\|\frac{\partial}{\partial t}(\nabla Y(X))\right\|_{L^{\infty}(0,T;H^{2})}
 +
  \left\|\frac{\partial}{\partial t}\left(\frac{\partial^2 Y_{l}}{\partial x_p\partial x_{i}}(X)\right)\right\|_{L^{\infty}(0,T; H^1)} \leq C. \label{X-3}
\end{gather}
\end{lem}
\begin{proof}
From \eqref{inverse:X} and the fact that $L^{\infty}(0,T; H^3)$ is an algebra, we deduce \eqref{X-6}. This yields in particular that 
\begin{equation}\label{18:04}
\|\nabla Y(X)\|_{L^{\infty}(0,T; H^3)} \leq C.
\end{equation}
Writing 
$$
\frac{\partial}{\partial y_m}\left( \frac{\partial Y_{l}}{\partial x_{i}}(X)\right)
=\sum_p \frac{\partial^2 Y_{l}}{\partial x_p\partial x_{i}}(X)\frac{\partial X_p}{\partial y_m}
$$
and using \eqref{18:04}, we deduce the estimate on $\frac{\partial^2 Y_{l}}{\partial x_p\partial x_{i}}(X)$.

From the expression \eqref{def:X}, we have $\frac{\partial}{\partial t}(\nabla X(t,\cdot))=Q(t)\nabla\widetilde{u}(t,\cdot)$, 
and using 
\begin{equation*}
\frac{\partial}{\partial t}(\nabla Y(X)) = - \nabla Y(X)\frac{\partial}{\partial t}(\nabla X)\nabla Y(X), 
\end{equation*}
we obtain the estimate of the second term in \eqref{X-3}.

Finally, we write
$$
\frac{\partial}{\partial y_m}\left[ \frac{\partial}{\partial t} \left(\frac{\partial Y_{l}}{\partial x_{i}}(X)\right)\right]
=\sum_p \frac{\partial}{\partial t} \left(\frac{\partial^2 Y_{l}}{\partial x_p\partial x_{i}}(X)\right) \frac{\partial X_p}{\partial y_m}
+\sum_p \frac{\partial^2 Y_{l}}{\partial x_p\partial x_{i}}(X) \frac{\partial}{\partial t} \left(\frac{\partial X_p}{\partial y_m}\right)
$$
and from the previous estimate, we have 
$$
\left\| \frac{\partial}{\partial y_m}\left[ \frac{\partial}{\partial t} \left(\frac{\partial Y_{l}}{\partial x_{i}}(X)\right)\right]
\right\|_{L^{\infty}(0,T;H^{1})}
+
\left\| 
\sum_p \frac{\partial^2 Y_{l}}{\partial x_p\partial x_{i}}(X) \frac{\partial}{\partial t} \left(\frac{\partial X_p}{\partial y_m}\right)
\right\|_{L^{\infty}(0,T;H^{1})}\leq C.
$$
Thus, using \eqref{18:04}, we deduce the estimate of the last term in \eqref{X-3}.
\end{proof}

Next we give some properties on $F_1, F_2, F_3, F_{4}$.
\begin{prop} \label{lip:F1}
There exist $\alpha>0$ and a positive constant $C$ depending on $R$, $k_p$, $k_d$, $\overline{\rho}$ and the other physical parameters, and 
on $\mc{F}(0)$ such that, for all $0 < T \leqslant T_3$, 
for all 
$$
(\widetilde{\rho},\widetilde{u}, \widetilde{\ell},\widetilde{\omega}),
(\widetilde{\rho}^{1},\widetilde{u}^{1},\widetilde{\ell}^{1},\widetilde{\omega}^{1}),
(\widetilde{\rho}^{2},\widetilde{u}^{2},\widetilde{\ell}^{2},\widetilde{\omega}^{2})
 \in \mc{S}_{T,R},$$
\begin{gather*}
\|F_{1}(\widetilde{\rho},\widetilde{u},\widetilde{\ell},\widetilde{\omega},Q)\|_{L^{2}(0,T;{H}^{3})\cap L^{\infty}(0,T;{H}^{2}) \cap H^{1}(0,T;{L}^{2})} \leqslant C T^{\alpha},\\
\|F_{2}(\widetilde{\rho},\widetilde{u},\widetilde{\ell},\widetilde{\omega},Q) \|_{H^1_{\infty}(0,T; L^2,H^2)} \leqslant 
C \left(T^{\alpha} + \|\omega_0 \times u_0\|_{H^1} + \|a\gamma\rho_0^{\gamma-2}\nabla \rho_0\|_{H^1}\right),\\
\|{F_{3}}(\widetilde{\rho},\widetilde{u},\widetilde{h},\widetilde{\ell},\widetilde{\omega},Q) \|_{H^{1}_{\infty}(0,T)} 
	\leqslant  C \left(T^{\alpha} +  |\omega_0 \times \ell_0| + |\ell_0| + \|\rho_0-\overline{\rho}\|_{H^1} + \|u_0\|_{H^3}\right),\\
\|{F_{4}}(\widetilde{\rho},\widetilde{u},\widetilde{\ell},\widetilde{\omega},Q) \|_{H^{1}_{\infty}(0,T)} \leqslant C T^{\alpha}, 
\end{gather*}
and
\begin{multline*}
\|{F_{1}}(\widetilde{\rho}^{1},\widetilde{u}^{1},\widetilde{\ell}^{1},\widetilde{\omega}^{1},Q^1) - {F_{1}}(\widetilde{\rho}^{2},\widetilde{u}^{2},\widetilde{\ell}^{2},\widetilde{\omega}^{2},Q^2)\|_{L^{2}(0,T;{H}^{3}) \cap L^{\infty}(0,T;{H}^{2}) \cap H^{1}(0,T;{L}^{2})}\\ \leqslant C T^{\alpha} \|(\widetilde{\rho}^{1},\widetilde{u}^{1},\widetilde{\ell}^{1},\widetilde{\omega}^{1}) - (\widetilde{\rho}^{2},\widetilde{u}^{2},\widetilde{\ell}^{2},\widetilde{\omega}^{2})\|_{\mc{S}_{T}},
\end{multline*}
\begin{equation*}
\|{F_{2}}(\widetilde{\rho}^{1},\widetilde{u}^{1},\widetilde{\ell}^{1},\widetilde{\omega}^{1},Q^1) - {F_{2}}(\widetilde{\rho}^{2},\widetilde{u}^{2},\widetilde{\ell}^{2},\widetilde{\omega}^{2},Q^2)\|_{H^1_{\infty}(0,T; L^2,H^2)}  \leqslant C T^{\alpha} \|(\widetilde{\rho}^{1},\widetilde{u}^{1},\widetilde{\ell}^{1},\widetilde{\omega}^{1}) - (\widetilde{\rho}^{2},\widetilde{u}^{2},\widetilde{\ell}^{2},\widetilde{\omega}^{2})\|_{\mc{S}_{T}},
\end{equation*} 
\begin{equation*}
\|{F_{3}}(\widetilde{\rho}^{1},\widetilde{u}^{1},\widetilde{h}^1,\widetilde{\ell}^{1},\widetilde{\omega}^{1},Q^1) - {F_{3}}(\widetilde{\rho}^{2},\widetilde{u}^{2},\widetilde{h}^2,\widetilde{\ell}^{2},\widetilde{\omega}^{2},Q^2)\|_{H^{1}_{\infty}(0,T)} \\ \leqslant C T^{\alpha} \|(\widetilde{\rho}^{1},\widetilde{u}^{1},\widetilde{\ell}^{1},\widetilde{\omega}^{1}) - (\widetilde{\rho}^{2},\widetilde{u}^{2},\widetilde{\ell}^{2},\widetilde{\omega}^{2})\|_{\mc{S}_{T}},
\end{equation*}
\begin{equation*}
\|{F_{4}}(\widetilde{\rho}^{1},\widetilde{u}^{1},\widetilde{\ell}^{1},\widetilde{\omega}^{1},Q^1) - {F_{4}}(\widetilde{\rho}^{2},\widetilde{u}^{2},\widetilde{\ell}^{2},\widetilde{\omega}^{2},Q^2)\|_{H^{1}_{\infty}(0,T)} \leqslant C T^{\alpha} \|(\widetilde{\rho}^{1},\widetilde{u}^{1},\widetilde{\ell}^{1},\widetilde{\omega}^{1}) - (\widetilde{\rho}^{2},\widetilde{u}^{2},\widetilde{\ell}^{2},\widetilde{\omega}^{2})\|_{\mc{S}_{T}}.
\end{equation*}
where $Q,\ Q^1,\ Q^2, \ \widetilde{h},\ \widetilde{h}^1, \ \widetilde{h}^2 \in H^3(0,T)$ are given by \eqref{13:43}.
\end{prop}
\begin{proof}
Using the definition \eqref{F1} of $F_1$, \eqref{15:06bis}, \eqref{13:40}, \eqref{ball for fixed point}, \eqref{X-6}
we have the following estimates
\begin{multline*}
\|F_{1}\|_{L^{2}(0,T;{H}^{3})} \leq C\|(\widetilde{\rho}+\overline{\rho})\|_{L^{\infty}(0,T; H^{3})} \|\nabla \widetilde{u}\|_{L^{2}(0,T;{H}^{3})} \|((\nabla Y)Q)^{\top} - \mathbb{I}_3\|_{L^{\infty}(0,T; H^{3})}
\\ 
+C\|(\widetilde{\rho} + \overline{\rho} - \rho_{0})\|_{L^{\infty}(0,T; H^3)}\|\operatorname{div}\widetilde{u}\|_{L^{2}(0,T;{H}^{3})} \leq C T^{\alpha},
\end{multline*}
\begin{multline}\label{time derivative F1}
\left\| \frac{\partial F_{1}}{\partial t} \right\|_{L^{2}(0,T;L^2)} \leq 
C\|(\widetilde{\rho}+\overline{\rho})\|_{L^{\infty}(0,T; H^{3})} 
\left\{ \left\|\nabla \frac{\partial\widetilde{u}}{\partial t} \right\|_{L^{2}(0,T;L^2)} \|((\nabla Y)Q)^{\top} - \mathbb{I}_3\|_{L^{\infty}(0,T; H^{3})}
\right. \\ \left.
+\|\nabla \widetilde{u}\|_{L^{2}(0,T;{H}^{3})} \left\| \frac{\partial}{\partial t}((\nabla Y(X))Q)^{\top} \right\|_{L^{\infty}(0,T;{H}^{2})}
\right\} 
\\ 
+CT^{1/2}\left\| \frac{\partial\widetilde{\rho}}{\partial t} \right\|_{L^{\infty}(0,T;H^2)}  \left\| \nabla\widetilde{u} \right\|_{L^{\infty}(0,T;{H}^{2})}
\left\| ((\nabla Y(X))Q)^{\top}\right\|_{L^{\infty}(0,T; H^{3})}
\\
+C\|\widetilde{\rho} + \overline{\rho} - \rho_{0}\|_{L^{\infty}(0,T;{H}^{3})}\left\|\operatorname{div} \frac{\partial\widetilde{u}}{\partial t}\right\|_{L^{2}(0,T;{H}^{1})}
\leq CT^{\alpha},
\end{multline}
\begin{multline*}
\|F_{1}\|_{L^{\infty}(0,T;{H}^{2})} \leq 
C\|\widetilde{\rho}+\overline{\rho}\|_{L^{\infty}(0,T; H^{3})} \|\nabla \widetilde{u}\|_{L^{\infty}(0,T;{H}^{2})} \|((\nabla Y)Q)^{\top} - \mathbb{I}_3\|_{L^{\infty}(0,T; H^{3})}
\\ 
+ C\|\widetilde{\rho} + \overline{\rho} - \rho_{0}\|_{L^{\infty}(0,T; H^3)}\|\operatorname{div}\widetilde{u}\|_{L^{\infty}(0,T;{H}^{2})} \leq CT^{1/2}.
\end{multline*}

Let us now estimate the $L^{2}(0,T;{H}^{2})$ norm of $F_{2}$. Here we only estimate some terms in \eqref{F2}, the other terms can be estimated similarly.
Using \eqref{rhoalpha}, \eqref{13:40}, \eqref{ball for fixed point}, \eqref{X-6}, \eqref{X-3}, 
\begin{multline*}\label{use:1/f est}
\left\|\frac{1}{\widetilde{\rho}+\overline{\rho}}\frac{\partial^{2}\widetilde{u}_{i}}{\partial y_{m}\partial y_{l}}\left(\frac{\partial Y_{m}}{\partial x_{p}}(X)\frac{\partial Y_{l}}{\partial x_{p}}(X) - \delta_{mp}\delta_{lp}\right)\right\|_{L^{2}(0,T;{H}^{2})} \\ \leq C\left\|\frac{1}{\widetilde{\rho}+\overline{\rho}}\right\|_{L^{\infty}(0,T;H^{3})} \left\|\frac{\partial^{2}\widetilde{u}_{i}}{\partial y_{m}\partial y_{l}}\right\|_{L^{2}(0,T;{H}^{2})}\left\|\frac{\partial Y_{m}}{\partial x_{p}}(X)\frac{\partial Y_{l}}{\partial x_{p}}(X) - \delta_{mp}\delta_{lp}\right\|_{L^{\infty}(0,T;{H}^{2})}\leq CT^{\alpha},
\end{multline*}
\begin{multline*}
\left\|\frac{1}{\widetilde{\rho}+\overline{\rho}} \frac{\partial \widetilde{u}_{i}}{\partial y_{l}}\frac{\partial^2 Y_{l}}{\partial x_{p}^2}(X)\right\|_{L^{2}(0,T;{H}^{2})}
\leq CT^{1/2}\left\|\frac{1}{\widetilde{\rho}+\overline{\rho}}\right\|_{L^{\infty}(0,T; H^{3})}\left\|\frac{\partial \widetilde{u}_{i}}{\partial y_{l}}\right\|_{L^{\infty}(0,T;{H}^{2})}\left\|\frac{\partial^2 Y_{l}}{\partial x_{p}^2}(X)\right\|_{L^{\infty}(0,T; {H}^{2})}
\\
\leq CT^{\alpha},
\end{multline*}
\begin{multline*}
\left\|(\widetilde{\rho}+\overline{\rho})^{\gamma - 2}\frac{\partial \widetilde{\rho}}{\partial y_{l}}\frac{\partial Y_{l}}{\partial x_{j}}(X)\right\|_{L^{2}(0,T;{H}^{2})}
\\
 \leq 
CT^{1/2} \|(\widetilde{\rho}+\overline{\rho})^{\gamma - 2}\|_{L^{\infty}(0,T; H^{2})} \left\|\frac{\partial \widetilde{\rho}}{\partial y_{l}}\right\|_{L^{\infty}(0,T;{H}^{2})}\left\|\frac{\partial Y_{l}}{\partial x_{j}}(X)\right\|_{L^{\infty}(0,T; H^{2})}\leq CT^{\alpha}.
\end{multline*}

For the estimate of the $H^{1}(0,T; L^2)$ norm of $F_{2}$, we also only give the estimates the $L^{2}(0,T; L^2)$ norm of some terms of the time derivative $F_{2}$. 
Again, the other terms can be estimated similarly.
First, we write
\begin{multline*}
\frac{\partial}{\partial{t}}\left[\frac{1}{\widetilde{\rho}+\overline{\rho}}\frac{\partial^{2}\widetilde{u}_{i}}{\partial y_{m}\partial y_{l}}\left(\frac{\partial Y_{m}}{\partial x_{p}}(X)\frac{\partial Y_{l}}{\partial x_{p}}(X) - \delta_{mp}\delta_{lp}\right)\right] \\
= -\frac{1}{(\widetilde{\rho}+\overline{\rho})^{2}}\frac{\partial \widetilde{\rho}}{\partial t}\frac{\partial^{2}\widetilde{u}_{i}}{\partial y_{m}\partial y_{l}}\left(\frac{\partial Y_{m}}{\partial x_{p}}(X)\frac{\partial Y_{l}}{\partial x_{p}}(X) - \delta_{mp}\delta_{lp}\right)\\
  + \left(\frac{1}{\widetilde{\rho}+\overline{\rho}}\right)\frac{\partial^{3}\widetilde{u}_{i}}{\partial t \partial y_{m}\partial y_{l}}\Bigg(\frac{\partial Y_{m}}{\partial x_{p}}(X)\frac{\partial Y_{l}}{\partial x_{p}}(X) 
- \delta_{mp}\delta_{lp}\Bigg) + \left(\frac{1}{\widetilde{\rho}+\overline{\rho}}\right)\frac{\partial^{2}\widetilde{u}_{i}}{\partial y_{m}\partial y_{l}}\frac{\partial}{\partial{t}}\left(\frac{\partial Y_{m}}{\partial x_{p}}(X)\frac{\partial Y_{l}}{\partial x_{p}}(X)\right),
\end{multline*}
\begin{multline*}
\frac{\partial}{\partial{t}}\left[\frac{1}{\widetilde{\rho}+\overline{\rho}} \frac{\partial \widetilde{u}_{i}}{\partial y_{l}}\frac{\partial^2 Y_{l}}{\partial x_{p}^2}(X)\right] = -\frac{1}{(\widetilde{\rho}+\overline{\rho})^2}\frac{\partial\widetilde{\rho}}{\partial t}\frac{\partial \widetilde{u}_{i}}{\partial y_{l}}\frac{\partial^2 Y_{l}}{\partial x_{p}^2}(X) + \frac{1}{\widetilde{\rho}+\overline{\rho}} \frac{\partial^{2} \widetilde{u}_{i}}{\partial t\partial y_{l}}\frac{\partial^2 Y_{l}}{\partial x_{p}^2}(X) 
\\
+ \frac{1}{\widetilde{\rho}+\overline{\rho}} \frac{\partial \widetilde{u}_{i}}{\partial y_{l}}\frac{\partial}{\partial{t}}\left(\frac{\partial^2 Y_{l}}{\partial x_{p}^2}(X)\right).
\end{multline*}
Using \eqref{rhoalpha}, \eqref{13:40}, \eqref{ball for fixed point}, \eqref{X-6}, \eqref{X-3}, we deduce that the above terms is estimated in $L^{2}(0,T; L^2)$  by $CT^{\alpha}$.

Finally, to obtain the $L^{\infty}(0,T; H^1)$ estimate of the term $F_2$, we use the following inequality \cite[Lemma 4.2]{MR3456813}: 
\begin{equation*}
\sup_{t\in (0,T)}\|F_2(t)\|_{H^1}\leq C\Big(\|F_2\|_{L^2(0,T;H^2)} + \|F_2\|_{H^1(0,T; L^2)}+ \|F_2(0)\|_{H^1}\Big),
\end{equation*} 
and since
\begin{equation*}\label{est:F2 initial}
\|F_2(0)\|_{H^1} \leq \|\omega_0 \times u_0\|_{H^1} + \|a\gamma\rho_0^{\gamma-2}\nabla \rho_0\|_{H^1},
\end{equation*}
we deduce the result for $F_2$.

It remains to estimate $F_{3}$ and $F_4$. We only consider $F_3$, the analysis for $F_4$ is the same.
From \eqref{F3}, we can see that the time derivative of $F_3$ involves the following terms (and similar ones)
\begin{multline*}
(\widetilde{\omega}\times \widetilde{\ell})', \quad (k_p Q^{\top}\widetilde{h})', \quad k_d \widetilde{\ell}',\quad
\int\limits_{\partial \mc{B}(0)}\Bigg(Q'\nabla \widetilde{u}\nabla Y(X) + Q\nabla \frac{\partial \widetilde{u}}{\partial t}\nabla Y(X) + Q\nabla \widetilde{u}\frac{\partial}{\partial t}\nabla Y(X)\Bigg)n\ d\Gamma
\\
-a\gamma\int\limits_{\partial \mc{B}(0)} (\overline{\rho}+\widetilde{\rho})^{\gamma-1}\frac{\partial \widetilde{\rho}}{\partial t}n\ d\Gamma.
\end{multline*}
Almost all the terms can be estimated in a direct way in $L^2(0,T)$ by using \eqref{est:h}, \eqref{rhoalpha}, \eqref{13:40}, \eqref{ball for fixed point}, \eqref{X-6}.
We have nevertheless to take care of 
$$
\int\limits_{\partial \mc{B}(0)} Q\nabla \frac{\partial \widetilde{u}}{\partial t}\nabla Y(X) n\ d\Gamma.
$$
For this term, we use standard interpolation result (see, for instance, \cite[Lemma A.5]{Boulakia_2019}) to obtain
$$
\left\| \nabla\frac{\partial \widetilde{u}}{\partial t}  \right\|_{L^{8/3}(0,T;H^{1/4})}
\leq C\left\| \nabla\frac{\partial \widetilde{u}}{\partial t}  \right\|_{L^{\infty}(0,T;L^{2})}^{1/4}
\left\| \nabla\frac{\partial \widetilde{u}}{\partial t}  \right\|_{L^{2}(0,T;H^{1})}^{3/4},
$$
where $C$ is independent of $T$.
Using a trace result and \eqref{13:40}, \eqref{ball for fixed point}, we deduce an estimate of $F_3'$ in $L^2(0,T)$ of the form $CT^{\alpha}$. 
To end the estimate of $F_3$, we use that
\begin{equation*}
\|F_3\|_{L^{\infty}(0,T)} \leq |F_3(0)| + T^{1/2}\|F_3\|_{H^{1}(0,T)}.
\end{equation*}
We have the following estimate:
\begin{equation*}
\left|\int\limits_{\partial \mc{B}(0)} (\overline{\rho}+\widetilde{\rho}(0))^{\gamma}n\ d\Gamma\right|=\left|\int\limits_{\partial \mc{B}(0)} \rho_0^{\gamma}n\ d\Gamma\right|
=\left|\int\limits_{\partial \mc{B}(0)} (\rho_0^{\gamma}-\overline{\rho}^{\gamma})n\ d\Gamma\right| \leq C \int\limits_{\partial \mc{B}(0)} |\rho_0-\overline{\rho}|\ d\Gamma.
\end{equation*}
Thus,
\begin{equation*}
|F_3(0)| \leq C\left( |\omega_0 \times \ell_0| + |\ell_0| + \|\rho_0-\overline{\rho}\|_{H^1} + \|u_0\|_{H^3} \right).
\end{equation*}

The estimates for the differences can be done in a similar way and we thus skip the corresponding proof.
\end{proof}

\subsection{Proof of \cref{local-in-time existence}}

\begin{proof} 
We are going to establish the local in time existence of \eqref{reform:fluiddensity}-\eqref{F4}. In order to do this we use a fixed-point argument. 

Assume $\overline{\rho}> 0$, $\delta_0$ satisfying the smallness assumptions introduced in the above section
and let us consider $(\rho_0,u_0,h_0,\ell_0,\omega_0)$ satisfying \eqref{initial cond req-1}, \eqref{initial cond req-2}. 
Recall that from \eqref{15:06}, we have
$\dfrac{\overline{\rho}}{2}\leq \rho_0 \leq \dfrac{3\overline{\rho}}{2}$ and thus, using Sobolev embeddings, there exists $C_1>0$ depending on $\overline{\rho}, \delta_0$ and the geometry such that
\begin{equation}\label{est:F20}
C \left(\|\omega_0 \times u_0\|_{H^1} + \|a\gamma\rho_0^{\gamma-2}\nabla \rho_0\|_{H^1}
+|\omega_0 \times \ell_0| + |\ell_0| + \|\rho_0-\overline{\rho}\|_{H^1} + \|u_0\|_{H^3}\right)
\leq C_1\widetilde{\delta}_0
\end{equation}
where $C$ is the constant appearing in \cref{lip:F1} and where we have set
\begin{equation*}
\widetilde{\delta}_0=\|\rho_0-\overline{\rho}\|_{H^3}+\|u_0\|_{H^3}+|h_1-h_0|+|\ell_0|+|\omega_0| \leq \delta_0.
\end{equation*}
We now fix $R>0$ as
\begin{equation}\label{radius of ball}
R = 2C_{L} C_1 \widetilde{\delta}_0,
\end{equation}
where $C_L$ is the continuity constant in estimate \eqref{est:linearsystemfull}. We take $T\leq T_3$, where $T_3=T_3(R)$ is the time obtained in the above section.

Let us define the following mapping 
\begin{align}\label{10:12}
\mathcal{N} : \quad &\mc{S}_{T,R} \rightarrow \mc{S}_{T,R}
\\ 
&(\widetilde{\rho}, \widetilde{u}, \widetilde{\ell}, \widetilde{\omega})  \mapsto (\widehat{\rho}, \widehat{u}, \widehat{\ell}, \widehat{\omega}).
\end{align}
For $(\widetilde{\rho}, \widetilde{u}, \widetilde{\ell}, \widetilde{\omega}) \in \mc{S}_{T,R}$, we define $X$ by \eqref{def:X}, $\widetilde{h}$ and $Q$ by \eqref{13:43}
and $F_1,\ F_2,\ F_3,\ F_4$ by \eqref{F1}-\eqref{F4}. Then $(\widehat{\rho}, \widehat{u}, \widehat{\ell}, \widehat{\omega})$ is the solution of 
\begin{align}
\frac{\partial \widehat{\rho}}{\partial t} + \rho_{0}\operatorname{div}\widehat{u} &= F_{1}(\widetilde{\rho}, \widetilde{u}, \widetilde{\ell}, \widetilde{\omega},Q)   \quad \mbox{ in }(0,T)\times \mc{F}(0), \label{fixedpoint:fluiddensity}\\
\frac{\partial \widehat{u}}{\partial t} - \frac{\mu}{\rho_{0}}\Delta \widehat{u} - \frac{\lambda+\mu}{\rho_{0}}\nabla \left(\operatorname{div} \widehat{u}\right)&=F_{2}(\widetilde{\rho}, \widetilde{u}, \widetilde{\ell}, \widetilde{\omega},Q)  \quad\mbox{ in }\quad (0,T)\times \mc{F}(0), \label{fixedpoint:fluidvelocity} \\ 
m\widehat{\ell}'  &=F_{3}(\widetilde{\rho}, \widetilde{u}, \widetilde{h}, \widetilde{\ell}, \widetilde{\omega},Q) \quad\mbox{ in }\quad (0,T),  \label{fixedpoint:rigidlinear}\\
J\widehat{\omega}' &=F_{4}(\widetilde{\rho}, \widetilde{u}, \widetilde{\ell}, \widetilde{\omega},Q) \quad\mbox{ in }\quad (0,T) \label{fixedpoint:rigidangular},
\end{align}
\begin{align}
\widehat{u}&=\widehat{\ell}+\widehat{\omega}\times (y-h_0) \quad\mbox{ on }\quad (0,T)\times \partial \mc{B}(0), \label{fixedpointboundary1:vel}\\
\widehat{u}&=0 \quad\mbox{ in }\quad (0,T)\times \partial \Omega. \label{fixedpointboundary2:vel}
\end{align}
\begin{align}
\widehat{\rho}(0,\cdot)&=\rho_0(\cdot) - \overline{\rho},\quad \widehat{u}(0,\cdot)=u_0(\cdot) \quad\mbox{ in }\quad \mc{F}(0), \label{fixedpointinitialcond:densityvel}\\
 \widehat{\ell}(0) &=\ell_{0}, \quad \widehat{\omega}(0) =\omega_0.\label{fixedpointbody:initial} 
\end{align}
In order to show that $\mathcal{N}$ is well-defined, we apply \cref{linear system:existence} to the above system.
First we note that \eqref{finalcompatibility-1}--\eqref{finalcompatibility-3} yield the compatibility conditions \eqref{comp cond1}--\eqref{comp cond3}. More precisely, the first condition is exactly condition \eqref{finalcompatibility-1}. 
Using the expression of $F_2$ in \eqref{F2}, we have
$$
\left[F_{2}(\widetilde{\rho},\widetilde{u},\widetilde{\ell},\widetilde{\omega},Q)\right](0,\cdot)
=
- \omega_0\times {u}_0 
+  \frac{1}{\rho_0} \nabla p_0,
$$
where $p_0=a\rho_0^\gamma$.
Thus, \eqref{finalcompatibility-2} yields the second condition.

On the other hand, using the expressions of $F_3$ and $F_4$ in \eqref{F3} and \eqref{F4}, we have
$$
\left[F_{3}(\widetilde{\rho},\widetilde{u},\widetilde{\ell},\widetilde{\omega},Q)\right](0,\cdot)
=-m({\omega}_0\times {\ell}_0) -\int\limits_{\partial \mc{B}(0)} \sigma(u_0,p_0) n\, d\Gamma - k_d\ell_0,
$$
$$
\left[F_{4}(\widetilde{\rho},\widetilde{u},\widetilde{\ell},\widetilde{\omega},Q)\right](0,\cdot)
= -\int\limits_{\partial \mc{B}(0)} (y-h_0) \times \sigma(u_0,p_0) n\, d\Gamma.
$$
These expressions of $F_3(0,\cdot)$ and $F_4(0,\cdot)$ show that \eqref{finalcompatibility-3} gives the third condition \eqref{comp cond3}.
We thus deduce from \cref{linear system:existence} the existence and uniqueness of 
$(\widehat{\rho},\widehat{u},\widehat{\ell},\widehat{\omega})\in\mc{S}_T$.
Combining \eqref{est:linearsystemfull}, \cref{lip:F1}, \eqref{est:F20} and \eqref{radius of ball}, we obtain
$$
\|(\widehat{\rho}, \widehat{u}, \widehat{\ell}, \widehat{\omega}) \|_{\mc{S}_{T}} 
\leqslant \frac{R}{2} + C T^{\alpha}.
$$
In particular, taking $T$ small enough, we deduce that $\mc{N}$ is well defined. 

Next we show that $\mathcal{N}$ is a contraction. 
Let $(\widetilde{\rho}^1, \widetilde{u}^1, \widetilde{\ell}^1, \widetilde{\omega}^1)$, 
$(\widetilde{\rho}^2, \widetilde{u}^2, \widetilde{\ell}^2, \widetilde{\omega}^2) \in \mc{S}_{T,R}.$ 
For $j=1,2,$ we set $\mathcal{N}(\widetilde{\rho}^j, \widetilde{u}^j, \widetilde{\ell}^j, \widetilde{\omega}^j):= (\widehat{\rho}^j, \widehat{u}^j, \widehat{\ell}^j, \widehat{\omega}^j).$ 
Using \cref{linear system:existence} and \cref{lip:F1}, we obtain 
\begin{equation*}
\| (\widehat{\rho}^1, \widehat{u}^1, \widehat{\ell}^1, \widehat{\omega}^1) - (\widehat{\rho}^2, \widehat{u}^2, \widehat{\ell}^2, \widehat{\omega}^2)\|_{\mc{S}_{T}}  \leqslant  C  T^{\alpha} \| (\widetilde{\rho}^1, \widetilde{u}^1,  \widetilde{\ell}^1, \widetilde{\omega}^1) - (\widetilde{\rho}^2, \widetilde{u}^2, \widetilde{\ell}^2, \widetilde{\omega}^2)\|_{\mc{S}_{T}}.
\end{equation*}Thus $\mathcal{N}$ is a contraction in $\mc{S}_{T,R}$ for $T$ small enough.

Finally, using \eqref{radius of ball} and \eqref{13:43}, we deduce 
\begin{multline*}
\|(\widetilde{\rho}, \widetilde{u}, \widetilde{\ell}, \widetilde{\omega}) \|_{\mc{S}_{T}} + \|h_1-h\|_{L^{\infty}(0,T)} \leqslant C \widetilde{\delta}_0=C\Big(\|\rho_0-\overline{\rho}\|_{H^3}+\|u_0\|_{H^3}+|h_1-h_0|+|\ell_0|+|\omega_0|\Big)
\end{multline*}
that yields \eqref{final local estimate}.
\end{proof}

\section{Global in time existence of solutions}\label{sec:Global in time existence of solution}

\subsection{A priori estimates}
We have already established a local-in-time existence result in \cref{local-in-time existence}. 
In order to obtain the global in time existence of the solutions, we need an appropriate a priori estimates. 
We recall that $\|\cdot\|_{\widehat{\mc{S}}_{0,T}}$ is introduced in \eqref{notation:norm}. We also introduce the following notation
to shorten the notation: for $Z=L^p$ or  $Z=W^{k,p}$, we set:
\begin{align*}
W^{k,\infty}_{T}(Z)=W^{k,\infty}(0,T; Z(\mc{F}(t))),\quad H^{k}_{T}(Z)=H^{k}(0,T; Z(\mc{F}(t))),\quad \mbox{for}\quad k=1,2,\\
W^{0,\infty}_{T}(Z)=L^{\infty}_{T}(Z)=L^{\infty}(0,T; Z(\mc{F}(t))),\quad H^{0}_{T}(Z)=L^2_{T}(Z)=L^2(0,T; Z(\mc{F}(t))).
\end{align*}

The main tool to prove the global in time existence of the solutions is the following proposition:
\begin{prop}\label{aprori est}
Let $h_1 \in \Omega^0$ and $\overline{\rho} >0$. 
Assume the feedback law \eqref{feedback} with $(k_p,k_d)$ satisfying \eqref{hypkp}. 
There exist $\varepsilon_0,C_0>0$ with $\varepsilon_0\leq \delta_0$ such that
if $(\rho, u, h, \ell, \omega)$ is a solution of system \eqref{continuity}--\eqref{feedback} with
\begin{equation}\label{smallness cond:apriori}
\|(\rho,u,\ell,\omega)\|_{\widehat{\mc{S}}_{0,T}} 
\leq  \varepsilon_0, 
\end{equation}
then the following estimate holds:
  \begin{equation}\label{est:apriori}
\|(\rho,u,\ell,\omega)\|_{\widehat{\mc{S}}_{0,T}} + \|\sqrt{k_p}(h_1-h)\|_{L^{\infty}(0,T)} \leq
C_{0}\Big(\|(\rho_0,u_0,\ell_0,\omega_0)\|_{\widehat{\mc{S}}_{0,0}}+|h_1-h_0|\Big).
\end{equation}
\end{prop}
\begin{proof}
The proof follows closely the idea of \cite[Proposition 8]{Muriel-Sergio-IHP}. We only repeat some parts of the proof to estimate $(h_1-h)$. 
We define 
\begin{equation*}
\rho^*(t,x) = \rho(t,x)-\overline{\rho}
\end{equation*}
and we rewrite \eqref{continuity}-\eqref{initial} as follows
\begin{equation}\label{full system:time dependent domain}
\left\{
        \begin{array}{ll}
        \displaystyle
        \frac{\partial \rho^*}{\partial t}+ u\cdot \nabla \rho^* + \overline{\rho}\operatorname{div}u=f_0(\rho^*,u,h,\omega) \quad & t \in (0,T), \, x\in\mc{F}(t),\\\displaystyle
\dfrac{\partial u}{\partial t} - \operatorname{div}\sigma^{*}(u,\rho^*)=f_1(\rho^*,u,h,\omega)\quad & t \in (0,T), \, x\in\mc{F}(t),\\
\overline{m}\ell'= -\displaystyle\int\limits_{\partial \mc{B}(t)} \sigma^{*}(u,\rho^*)N\, d\Gamma + \overline{k_p}(h_1-h(t)) - \overline{k_d} \ell(t) + f_2(\rho^*,u,h,\omega)&  t \in (0,T),\\ 
\overline{J}\omega'=-\displaystyle\int\limits_{\partial \mc{B}(t)} (x-h) \times \sigma^{*}(u,\rho^*)N\, d\Gamma + f_3(\rho^*,u,h,\omega) &  t \in (0,T),\\
h'=\ell & t \in (0,T),\\
u(t,x)=0,\quad &  t \in (0,T),\, x \, \in \partial \Omega, \\
u(t,x)= \ell(t)+\omega(t) \times (x-h(t)), \quad &  t \in (0,T),\, x \, \in \partial \mc{B}(t),\\
\rho^*(0,\cdot)=\rho_0-\overline{\rho}, \quad u(0,\cdot)=u_0 \quad \mbox{in}\quad \mc{F}(0),\\
h(0)=h_0,\quad \ell(0)=\ell_0,\quad \omega(0)=\omega_0,
        \end{array}
        \right.
\end{equation}
In the above system \eqref{full system:time dependent domain} 
$$
\overline{m}=\frac{m}{\overline{\rho}}, \quad
\overline{J}=\frac{J}{\overline{\rho}}, \quad
\overline{k_p}=\frac{k_p}{\overline{\rho}}, \quad
\overline{k_d}=\frac{k_d}{\overline{\rho}}, \quad
\overline{\mu}=\frac{\mu}{\overline{\rho}}, \quad
\overline{\lambda}=\frac{\lambda}{\overline{\rho}}, 
$$
\begin{equation*}
\sigma^{*}(u,\rho^*)= 2\overline{\mu} \mathbb{D}(u) + \overline{\lambda} \div u \mathbb{I}_3-p^{*}\rho^*\mathbb{I}_3,\quad 
p^{*}=a\gamma\overline{\rho}^{\gamma-2},
\end{equation*}
and
\begin{equation*}
\left\{
        \begin{array}{ll}
        \displaystyle      
f_0(\rho^*,u,h,\omega)&= -\rho^*\operatorname{div}u,\\
f_1(\rho^*,u,h,\omega)&= -\left(u\cdot \nabla\right)u-\left(\dfrac{1}{\overline{\rho}}-\dfrac{1}{\rho^*+\overline{\rho}}\right)\operatorname{div}\left(2\mu \mathbb{D}(u) + \lambda \operatorname{div}u\mathbb{I}_3\right) 
\\
&\quad+ \left(p^{*} - a\gamma(\rho^*+\overline{\rho})^{\gamma-2}\right)\nabla \rho^*,\\
f_2(\rho^*,u,h,\omega)&= -\displaystyle\int\limits_{\partial \mc{B}(t)}\left(p^{*}\rho^* - \frac{a(\rho^*+\overline{\rho})^{\gamma}}{\overline{\rho}}\right)N\,d\Gamma,\\
f_3(\rho^*,u,h,\omega)&= - \displaystyle\int\limits_{\partial \mc{B}(t)}(x-h) \times \left(\left(p^{*}\rho^* - \frac{a(\rho^*+\overline{\rho})^{\gamma}}{\overline{\rho}}\right)N\right)\,d\Gamma.      
        \end{array}
        \right.
\end{equation*}

We take $\varepsilon_0$ small enough in \eqref{smallness cond:apriori} so that 
\begin{equation*}
\rho^* + \overline{\rho} \geq \frac{\overline{\rho}}{2}.
\end{equation*}  

After some calculations (that we skipped here), we obtain
\begin{multline*}
\|f_0\|^2_{L^2_{T}(H^3)} + \|f_0\|^2_{L^{\infty}_{T}(H^2)}+\|f_0\|^2_{H^1_{T}(L^2)}+\|f_1\|^2_{L^2_{T}(H^2)} + \|f_1\|^2_{L^{\infty}_{T}(H^1)}+\|f_1\|^2_{H^1_{T}(L^2)}\\+\|f_2\|^2_{H^1(0,T)}+\|f_3\|^2_{H^1(0,T)} \leq C\|(\rho,u,\ell,\omega)\|_{\widehat{\mc{S}}_{0,T}}^4,
\end{multline*}
and
\begin{multline*}
\left\|\dfrac{\partial \rho^*}{\partial t}(0,\cdot)\right\|^2_{L^2}+\left\|\dfrac{\partial u}{\partial t}(0,\cdot)\right\|^2_{H^1}+|\ell'(0)|^2
+|\omega'(0)|^2
\\
\leq
C\left(
\|\rho_0-\overline{\rho}\|^2_{H^2}+\|u_0\|^2_{H^3} + |\ell_0|^2+|h_0-h_1|^2
+ \|f_0\|^2_{L^{\infty}_{T}(L^2)}
+ \|f_1\|^2_{L^{\infty}_{T}(H^1)}
 + \|f_2\|^2_{L^{\infty}_{T}}
+ \|f_3\|^2_{L^{\infty}_{T}}
\right).
 \end{multline*}

In particular, if we can show
\begin{multline}\label{required estimate}
\|(\rho,u,\ell,\omega)\|_{\widehat{\mc{S}}_{0,T}}^2 + \|\sqrt{k_p}(h_1-h)\|_{L^{\infty}(0,T)}^2 \leq C\Big(\|f_0\|^2_{L^2_{T}(H^3)} + \|f_0\|^2_{L^{\infty}_{T}(H^2)}+\|f_0\|^2_{H^1_{T}(L^2)}+\|f_1\|^2_{L^2_{T}(H^2)}\\
 + \|f_1\|^2_{L^{\infty}_{T}(H^1)}+\|f_1\|^2_{H^1_{T}(L^2)}+\|f_2\|^2_{H^1(0,T)}+\|f_3\|^2_{H^1(0,T)}+ \left\|\dfrac{\partial \rho^*}{\partial t}(0,\cdot)\right\|^2_{L^2}+\|\rho_0-\overline{\rho}\|^2_{H^3} +\left\|\dfrac{\partial u}{\partial t}(0,\cdot)\right\|^2_{H^1}\\
+\|u_0\|^2_{H^3(\mc{F}(0))}
+|h_0-h_1|^2 + |\ell'(0)|^2+|\ell_0|^2+|\omega'(0)|^2+|\omega_0|^2+\|(\rho,u,\ell,\omega)\|_{\widehat{\mc{S}}_{0,T}}^3 \\
+\|(\rho,u,\ell,\omega)\|_{\widehat{\mc{S}}_{0,T}}^4 
\Big),
\end{multline}
then
 \begin{multline*}
 \|(\rho,u,\ell,\omega)\|_{\widehat{\mc{S}}_{0,T}}^2 + \|\sqrt{k_p}(h_1-h)\|_{L^{\infty}(0,T)}^2 \leq C\Big(\|(\rho,u,\ell,\omega)\|_{\widehat{\mc{S}}_{0,T}}^3 
 +\|(\rho,u,\ell,\omega)\|_{\widehat{\mc{S}}_{0,T}}^4 
 \\
 +\|\rho_0-\overline{\rho}\|^2_{H^3}+\|u_0\|^2_{H^3}+|h_0-h_1|^2+ |\ell_0 |^2+|\omega_0|^2\Big).
 \end{multline*}
The condition \eqref{smallness cond:apriori} with $\varepsilon_0$ small enough combined with the above relation yields 
\eqref{est:apriori}. The proof of \eqref{required estimate} is done below.
\end{proof}

The proof of \eqref{required estimate} (that is necessary to finish the proof of \cref{aprori est})
is done in a precise way in \cite[Section 4]{Muriel-Sergio-IHP} in the case $k_p=0$ and $k_d=0$. 
The presence of the corresponding terms only changes the two lemmas on time regularity (Lemma 13 and Lemma 14 in \cite{Muriel-Sergio-IHP}). 
Here we state these two lemmas in our case and give the idea of their proofs with a particular attention to the feedback term.
Then using these two lemmas and the elliptic results \cite[Section 4]{Muriel-Sergio-IHP}, we can deduce \eqref{required estimate} and thus end the proof of \cref{aprori est}.
\begin{lem}\label{time reg1}
Let k=0,1. For every $\varepsilon > 0$, there exists a constant $C>0$ such that 
\begin{multline}\label{sup-norm-estimate}
\|\rho^*\|_{W^{k,\infty}_{T}(L^2)} + \|u\|_{H^k_{T}(H^1)}  + \|u\|_{{W^{k,\infty}_{T}(L^2)}} + \|\ell\|_{W^{k,\infty}(0,T)} + \|\ell\|_{H^k(0,T)} +\|\omega\|_{W^{k,\infty}(0,T)} \\
+ \|\sqrt{k_p}(h_1-h)\|_{L^{\infty}(0,T)}
 \leq \varepsilon\Big(\|\rho^*\|_{H^k_{T}(L^2)}+\|\ell\|_{H^k(0,T)} + \|\omega\|_{H^k(0,T)} \Big) \\
 + C \Big(\|f_0\|_{H^k_{T}(L^2)}+ \|f_1\|_{H^k_{T}(L^2)}+\|f_2\|_{H^k(0,T)}
	 +\|f_3\|_{H^k(0,T)}
	 \\
	 +\|\rho_0-\overline{\rho}\|_{L^2}+\|u_0\|_{L^2}+|h_1-h_0| + |\ell_0 |+|\omega_0| \\
 +\left\|\frac{\partial\rho^*}{\partial t}(0,\cdot)\right\|_{L^2}+\left\|\frac{\partial u}{\partial t}(0,\cdot)\right\|_{L^2}
+|\ell'(0) |+|\omega'(0)|+\|(\rho,u,\ell,\omega)\|_{\widehat{\mc{S}}_{0,T}}^{3/2} + \|(\rho,u,\ell,\omega)\|_{\widehat{\mc{S}}_{0,T}}^2 \Big).
\end{multline}
\end{lem}
\begin{proof}[Proof of \cref{time reg1}]
\underline{Case $k=0.$} We multiply equation \eqref{full system:time dependent domain}$_1$ by $p^{*}\rho^*/\overline{\rho}$, \eqref{full system:time dependent domain}$_2$ by $u$, 
\eqref{full system:time dependent domain}$_3$ by $\ell$ and \eqref{full system:time dependent domain}$_4$ by $\omega$:
\begin{multline}\label{expression1}
\int\limits_{\mc{F}(t)}\left(\frac{p^{*}}{2\overline{\rho}}|\rho^*|^2+\frac{|u|^2}{2}\right)\,dx 
+ \int\limits_{0}^{t}\int\limits_{\mc{F}(s)}\left(2\overline{\mu} |\mathbb{D}(u)|^2 + \overline{\lambda} |\operatorname{div}u|^2\right)\,dx\,ds 
\\
+ \frac{\overline m}{2}|\ell|^2 + \frac{\overline J}{2}|\omega|^2 + \frac{\overline k_p}{2}(t)|h_1-h(t)|^2 + \overline k_d \int\limits_{0}^{t} |\ell|^2\,ds 
\\
= \int\limits_{0}^{t}\int\limits_{\mc{F}(s)} (f_0p^{*}\rho^*/\overline{\rho}+f_1\cdot u)\,dx\,ds + \int\limits_{0}^{t}(f_2\cdot\ell + f_3\cdot\omega)\,ds 
\\
+ \int\limits_{0}^{t}\int\limits_{\mc{F}(s)}\Big(\frac{p^{*}}{2\overline{\rho}}|\rho^*|^2 \operatorname{div}u 
+  \operatorname{div}\left(\frac{|u|^2u}{2}\right)\Big)\,dx\,ds 
\\
+ \int\limits_{0}^{t}\frac{\overline k_{p}'}{2}(s)|h_1-h(s)|^2\,ds 
+ \int\limits_{\mc{F}(0)}\frac{p^{*}}{2\overline{\rho}}|\rho_0-\overline{\rho}|^2\,dy 
+ \int\limits_{\mc{F}(0)}\frac{|u_0|^2}{2}\,dy + \frac{\overline m}{2}|\ell_0|^2 + \frac{\overline J}{2}|\omega_0|^2.
\end{multline}
Following standard calculation, we have
\begin{equation}\label{1-RHS3-1624}
\int\limits_{0}^{t}\int\limits_{\mc{F}(s)}\left(\frac{p^{*}}{2\overline{\rho}}|\rho^*|^2 \operatorname{div}u +  \operatorname{div}\left(\frac{|u|^2u}{2}\right)\right)\,dx\,ds \leq C\|(\rho,u,\ell,\omega)\|_{\widehat{\mc{S}}_{0,T}}^3 .
\end{equation}
It only remains to estimate 
\begin{multline}\label{est:springterm}
\int\limits_{0}^{T}\frac{\overline k_{p}'}{2}(s)|h_1-h(s)|^2\,ds \leq \int\limits_{0}^{T_I}\frac{\overline k_{p}'}{2}(s)|h_1-h(s)|^2\,ds \\
\leq \|\overline k_{p}'\|_{L^{\infty}(0,T)}\Bigg(T_I|h_1-h_0|^2 + \int\limits_{0}^{T_{I}} \left(\int\limits_0^s \ell(z)\, dz\right)^2\,ds \Bigg). 
\end{multline}
By using H\"{o}lder's inequality and \eqref{hypkp}, we obtain
\begin{equation}\label{1-RHS4}
\int\limits_{0}^{T}\frac{\overline k_{p}'}{2}(s)|h_1-h(s)|^2\,ds \leq C|h_1-h_0|^2 + \frac{k_d}{2} \int\limits_{0}^{T} |\ell|^2\,ds.
\end{equation}
Combining \eqref{expression1}, \eqref{1-RHS3-1624}, \eqref{1-RHS4} and Young's inequality, we deduce the result for $k=0$.

\underline{Case $k=1.$} By differentiating \eqref{full system:time dependent domain} with respect to 
$t$, we obtain:
\begin{equation}\label{full system:time derivative}
\left\{
        \begin{array}{ll}
        \displaystyle
        \frac{\partial}{\partial t}\left(\frac{\partial \rho^*}{\partial t}\right)+ (u\cdot \nabla)\frac{\partial \rho^*}{\partial t} + \overline{\rho}\operatorname{div}\frac{\partial u}{\partial t}=G_0 \quad & t \in (0,T), \, x\in\mc{F}(t),\\
\dfrac{\partial}{\partial t}\left(\dfrac{\partial u}{\partial t}\right) - \operatorname{div}\sigma^{*}\left(\dfrac{\partial u}{\partial t},\dfrac{\partial \rho^*}{\partial t}\right)=G_1\quad & t \in (0,T), \, x\in\mc{F}(t),\\
\overline m\ell''= -\displaystyle\int\limits_{\partial \mc{B}(t)} \sigma^{*} \left(\dfrac{\partial u}{\partial t},\dfrac{\partial \rho^*}{\partial t}\right)N\, d\Gamma 
+ [\overline k_p(h_1-h(t))]' - \overline k_d \ell'(t) + G_2&  t \in (0,T),\\ 
\overline J\omega''=-\displaystyle\int\limits_{\partial \mc{B}(t)} (x-h) \times \sigma^{*} \left(\dfrac{\partial u}{\partial t},\dfrac{\partial \rho^*}{\partial t}\right)N\, d\Gamma + G_3 &  t \in (0,T),\\
h'=\ell & t \in (0,T),\\
\dfrac{\partial u}{\partial t}(t,x)=0,\quad &  t \in (0,T),\, x \, \in \partial \Omega, \\
\dfrac{\partial u}{\partial t}(t,x)= \ell'(t)+\omega'(t) \times (x-h(t))+G_4, \quad &  t \in (0,T),\, x \, \in \partial \mc{B}(t).
        \end{array}
        \right.
\end{equation}
where
\begin{equation*}
\left\{
        \begin{array}{ll}
        \displaystyle      
&G_0=\dfrac{\partial f_0}{\partial t} - \left(\dfrac{\partial u}{\partial t}\cdot \nabla\right)\rho^*,\quad G_1=\dfrac{\partial f_1}{\partial t},\quad G_2=\dfrac{\partial f_2}{\partial t} - \displaystyle\int\limits_{\partial \mc{B}(t)} \ell \cdot \nabla(\sigma^*(u,\rho^*)N)\, d\Gamma,\\
&G_3= \dfrac{\partial f_3}{\partial t} - \displaystyle\int\limits_{\partial \mc{B}(t)}\ell \cdot \nabla((x-h)\times (\sigma^*(u,\rho^*))N)\,d\Gamma + \displaystyle\int\limits_{\partial \mc{B}(t)} \ell \times (\sigma^*(u,\rho^*))N)\,d\Gamma,\\    
&G_4=-(\ell\cdot\nabla)u.
        \end{array}
        \right.
\end{equation*}
As in the first case, we multiply 
equation \eqref{full system:time derivative}$_{1}$ by $\dfrac{p*}{\overline{\rho}}\dfrac{\partial \rho^*}{\partial t}$,
equation \eqref{full system:time derivative}$_{2}$ by $\dfrac{\partial u}{\partial t}$, 
equation \eqref{full system:time derivative}$_{3}$ by $\ell'$, 
and 
equation \eqref{full system:time derivative}$_{4}$ by $\omega'$. After some computations, we find
\begin{multline}\label{expression2}
\int\limits_{\mc{F}(t)}\left(\frac{p^{*}}{2\overline{\rho}}\left|\dfrac{\partial \rho^*}{\partial t}\right|^2 + \frac{1}{2} \left|\frac{\partial u}{\partial t}\right|^2\right)\,dx + \int\limits_{0}^{t}\int\limits_{\mc{F}(s)} \left(2\overline \mu\left|\mathbb{D}\left(\frac{\partial u}{\partial t}\right)\right|^2 +  \overline \lambda \left|\operatorname{div}\frac{\partial u}{\partial t}\right|^2\right)\,dx\,ds 
\\
+ \frac{\overline m}{2}|\ell'(t)|^2 + \frac{\overline J}{2}|\omega'(t)|^2 
+ \overline k_d \int\limits_{0}^{t} |\ell'(s)|^2\,ds 
\\
= \int\limits_{0}^{t}\int\limits_{\mc{F}(s)} \left(G_0\frac{p^{*}}{\overline{\rho}}\dfrac{\partial \rho^*}{\partial t}+G_1\cdot \dfrac{\partial u}{\partial t}\right)\,dx\,ds + \int\limits_{0}^{t}(G_2\cdot\ell' + G_3\cdot\omega')\,ds 
+ \int\limits_0^t\displaystyle\int\limits_{\partial \mc{B}(s)} G_4 \cdot \sigma^{*} \left(\dfrac{\partial u}{\partial t},\dfrac{\partial \rho^*}{\partial t}\right)N\,d\Gamma\,ds \\
+ \int\limits_{0}^{t}\int\limits_{\mc{F}(s)}\frac{p^{*}}{2\overline{\rho}}\left|\dfrac{\partial \rho^*}{\partial t}\right|^2 \operatorname{div} u\,dx\,ds 
+ \int\limits_{0}^{t}\int\limits_{\partial \mc{F}(s)}\frac{1}{2}\operatorname{div}\left(\left|\dfrac{\partial u}{\partial t}\right|^2 u\right)\,dx\,ds 
\\
+ \int\limits_0^t [\overline k_p(s)(h_1-h(s))]'\cdot \ell'(s) \,ds 
+ \int\limits_{\mc{F}(0)}\left(\frac{p^{*}}{2\overline{\rho}}\left|\dfrac{\partial \rho^*}{\partial t}(0)\right|^2 
+ \frac{1}{2} \left|\frac{\partial u}{\partial t}(0)\right|^2\right)\,dy + \frac{\overline m}{2}|\ell'(0)|^2 + \frac{\overline J}{2}|\omega'(0)|^2 .
\end{multline}

We have the following estimates as in \cite[Lemma 13]{Muriel-Sergio-IHP}: 
\begin{multline}\label{19:32}
\|G_0\|_{L^2_T(L^2)}^2+\|G_1\|_{L^2_T(L^2)}^2
+\left\|G_2\right\|^2_{L^2(0,T)}
+\left\|G_3\right\|^2_{L^2(0,T)}
+\left\|G_4\right\|^2_{L^2_T(L^2(\partial \mc{B}(t))}
\\
+\int\limits_{0}^{t}\int\limits_{\mc{F}(s)}\left(\left|\dfrac{\partial \rho^*}{\partial t}\right|^2 \operatorname{div} u
+ \operatorname{div}\left(\left|\dfrac{\partial u}{\partial t}\right|^2 u\right)\right)\,dx\,ds 
\\
\leq 
C \left(
\left\|\frac{\partial f_0}{\partial t}\right\|^2_{L^2_{T}(L^2)} 
+\left\|\frac{\partial f_1}{\partial t}\right\|^2_{L^2_{T}(L^2)}
+\left\|\frac{\partial f_2}{\partial t}\right\|^2_{L^2(0,T)}
+\left\|\frac{\partial f_3}{\partial t}\right\|^2_{L^2(0,T)}
\right.\\ \left.
+\|(\rho,u,\ell,\omega)\|_{\widehat{\mc{S}}_{0,T}}^3 
+ \|(\rho,u,\ell,\omega)\|_{\widehat{\mc{S}}_{0,T}}^4 
\right)
\end{multline}
It only remains to estimate the term coming from the feedback:
\begin{multline*}
\int\limits_0^t [\overline k_p(s)(h_1-h(s))]'\cdot \ell'(s)\,ds 
= 
\int\limits_0^t \overline k_p'(s)(h_1-h(s))\cdot \ell'(s)\,ds 
- \int\limits_0^t \overline k_p(s)\ell(s) \cdot \ell'(s)\,ds 
\\
= \int\limits_0^t \overline k_p'(s)(h_1-h(s))\cdot \ell'(s)\,ds + \int\limits_0^t \frac{\overline k_p'}{2}(s)|\ell(s)|^2\,ds - \frac{\overline k_p}{2}(t)|\ell(t)|^2,
\end{multline*}
and proceeding as in \eqref{est:springterm}, we have the following estimates
\begin{equation}\label{2-RHS8}
\int\limits_0^t [\overline k_p(s)(h_1-h(s))]'\cdot \ell'(s)\,ds \leq C\left(|h_1-h_0|^2 + \int\limits_0^T |\ell(s)|^2\, ds\right) + \frac{\overline k_d}{2}\int\limits_{0}^T |\ell'(s)|^2\, ds. 
\end{equation}
We can estimate $\|\ell\|_{L^2(0,T)}^2$ with  \eqref{sup-norm-estimate} for $k=0.$
With this remark and combining  \eqref{expression2}, inequality \eqref{2-RHS8} and the above estimates we deduce \eqref{sup-norm-estimate} for $k=1.$
\end{proof}

\begin{lem}\label{time reg2}
Let $k=0,1$. There exists a constant $C>0$ such that 
\begin{multline}\label{est:time reg2}
\left\|\frac{\partial \rho^*}{\partial t}\right\|_{H^{k}_{T}(L^2)} + \|u\|_{W^{k,\infty}_{T}(H^1)}  
+ \left\|\frac{\partial u}{\partial t}\right\|_{{H^k_{T}(L^2)}} + \|\ell'\|_{H^k(0,T)}+\|\omega'\|_{H^k(0,T)} 
\\
\leq C\Big(\|\rho^*\|_{W^{k,\infty}_T(L^2)} + \|u\|_{H^k_T(H^1)}+\|\ell\|_{W^{k,\infty}(0,T)}+ \|\ell\|_{H^k(0,T)}
+ \|\sqrt{k_p}(h_1-h)\|_{L^{\infty}(0,T)}
\\
+\|f_0\|_{H^k_{T}(L^2)} + \|f_1\|_{H^k_{T}(L^2)}+\|f_2\|_{H^k(0,T)}+\|f_3\|_{H^k(0,T)}\\
+\|\rho_0-\overline{\rho}\|_{L^2}+\left\|\frac{\partial \rho}{\partial t}(0,\cdot)\right\|_{L^2}+\|u_0\|_{H^1}+\left\|\frac{\partial u}{\partial t}(0,\cdot)\right\|_{H^1}+|h_1-h_0|
\\
+ |\ell_0|+ |\ell'(0)|+ |\omega_0|
+\|(\rho,u,\ell,\omega)\|_{\widehat{\mc{S}}_{0,T}}^{3/2} + \|(\rho,u,\ell,\omega)\|_{\widehat{\mc{S}}_{0,T}}^2 \Big).
\end{multline}
\end{lem}
\begin{proof}[Proof of \cref{time reg2}]
\underline{Case $k=0$.} 
We multiply equation \eqref{full system:time dependent domain}$_1$ by $\dfrac{\partial\rho^*}{\partial t}$, 
\eqref{full system:time dependent domain}$_2$ by $\dfrac{\partial u}{\partial t}$,
\eqref{full system:time dependent domain}$_3$ by $\ell'$ and
\eqref{full system:time dependent domain}$_4$ by $\omega'$. After standard computations, we find
\begin{multline}\label{expression 3}
\int\limits_{0}^{t}\int\limits_{\mc{F}(s)} \left(\left|\dfrac{\partial \rho^*}{\partial t}\right|^2 + \left|\dfrac{\partial u}{\partial t}\right|^2\right)\,dx\,ds 
+ \int\limits_{\mc{F}(t)} \left(2\overline \mu|\mathbb{D}(u)|^2 +   \overline \lambda|\operatorname{div}u|^2\right)\,dx 
+ \int\limits_{0}^{t}\left(\overline{m} |\ell'|^2 + \overline{J} |\omega'|^2\right) \ ds
\\ 
+\overline k_d |\ell(t)|^2
	-2 \int_0^t \overline k_p(h_1-h)\cdot \ell' \ ds
\\
\leq 
C\Bigg(\int\limits_{0}^{t}\int\limits_{\mc{F}(s)} |\operatorname{div}(\left(2\overline \mu|\mathbb{D}(u)|^2 +   \overline \lambda|\operatorname{div}u|^2\right)u)|\,dx\,ds 
+ \int\limits_{0}^{t}\int\limits_{\mc{F}(s)}\left(p^{*}\rho^*\operatorname{div}\frac{\partial u}{\partial t}
	-\overline{\rho}\frac{\partial \rho^*}{\partial t}\operatorname{div}u\right)\,dx\,ds 
\\
+\int\limits_{0}^{t}\int\limits_{\mc{F}(s)} \left|\frac{\partial \rho^*}{\partial t}\right||(u\cdot\nabla \rho^*)|\,dx\,ds 
+ \int\limits_{\mc{F}(0)} |\nabla u_0|^2\,dx +\overline k_d|\ell_0|^2
\\
+ \int\limits_{0}^{t}\left(\int\limits_{\mc{F}(s)} \left(|f_0|^2+|f_1|^2 \right)\ dx+ |f_2|^2+|f_3|^2\right)\ ds 
+ \int\limits_{0}^{t}\int\limits_{\partial \mc{B}(s)} G_4 \cdot \sigma^* \left(u,\rho^*\right)N\,d\Gamma\,ds\Bigg).
\end{multline}
The terms in the right-hand side of \eqref{expression 3} can be estimated as in Lemma 14 in \cite{Muriel-Sergio-IHP}. We only estimate
\begin{equation}\label{for arnab}
-\int_0^t \overline k_p(h_1-h)\cdot \ell' \ ds
=
- \overline k_p(t) (h_1-h)\cdot \ell(t)
+\int\limits_{0}^{t} \left(\overline k_p'(h_1-h)\cdot \ell 
-\overline k_p |\ell|^2\right)\, ds
\end{equation}
and thus
$$
\left|-\int_0^t \overline k_p(h_1-h)\cdot \ell' \ ds
\right|
 \leq C\left(|h_1-h_0|^2 +k_p |h_1-h|^2
 +|\ell|^2+ \int\limits_0^t |\ell(s)|^2\, ds\right).
$$

\underline{Case $k=1$.} 
We multiply \eqref{full system:time derivative}$_1$ by $\dfrac{\partial^2\rho^*}{\partial t^2}$, 
\eqref{full system:time derivative}$_2$ by $\dfrac{\partial^2 u}{\partial t^2}$
\eqref{full system:time derivative}$_3$ by $\ell''$
and
\eqref{full system:time derivative}$_4$ by $\omega''$.
Following the proof of Lemma 14 in \cite{Muriel-Sergio-IHP}, we find
\begin{multline}\label{expression 4}
\int\limits_{0}^{t}\int\limits_{\mc{F}(s)} \left(\left|\dfrac{\partial^2 \rho^*}{\partial t^2}\right|^2 
	+ \left|\dfrac{\partial^2 u}{\partial t^2}\right|^2\right)\,dx\,ds 
+ \int\limits_{\mc{F}(t)} \left(2\overline \mu\left|\mathbb{D}\left(\dfrac{\partial u}{\partial t}\right)\right|^2 
	+ \overline \lambda\left|\operatorname{div}\dfrac{\partial u}{\partial t}\right|^2\right)\,dx 
\\
+\int\limits_{0}^{t}\left(\frac{\overline m}{2}|\ell''|^2+\frac{\overline J}{2}|\omega''|^2\right)\,ds +2 \overline k_d (|\ell'(t)|^2-|\ell'(0)|^2)
\\
 - \int\limits_{0}^{t} \overline k_p'(s)(h_1-h(s))\cdot \ell''(s)\, ds + \int\limits_{0}^{t} \overline k_p(s) \ell(s)\cdot \ell''(s)\,ds
\\
\leq C\Big(\|\rho^*\|_{W^{1,\infty}_T(L^2)}^2 + \|u\|_{H^1_T(H^1)}^2+\|\ell\|_{W^{1,\infty}(0,T)}^2+ \|\ell\|_{H^1(0,T)}^2
\\
+\|f_0\|_{H^1_{T}(L^2)}^2 + \|f_1\|_{H^1_{T}(L^2)}^2+\|f_2\|_{H^1(0,T)}^2+\|f_3\|_{H^1(0,T)}^2
\\
+\|\rho_0-\overline{\rho}\|_{L^2}^2
+\left\|\frac{\partial \rho}{\partial t}(0,\cdot)\right\|_{L^2}^2
+\|u_0\|_{H^1}^2
+\left\|\frac{\partial u}{\partial t}(0,\cdot)\right\|_{H^1}^2
\\
+ |\ell_0|^2+ |\ell'(0)|^2+ |\omega_0|^2
+\|(\rho,u,\ell,\omega)\|_{\widehat{\mc{S}}_{0,T}}^{3} + \|(\rho,u,\ell,\omega)\|_{\widehat{\mc{S}}_{0,T}}^4 \Big).
\end{multline}
We estimate the additional term due to the feedback:
\begin{multline}\label{derivative of spring term}
\left|\int\limits_{0}^{t} (-\overline k_p'(s)(h_1-h(s))\cdot \ell''(s) + \overline k_p(s) \ell(s)\cdot \ell''(s))\, ds\right| 
\\
\leq C\left(|h_1-h_0|^2 + \int\limits_0^T |\ell(s)|^2\, ds\right)  + \frac{\overline m}{4} \int\limits_{0}^{t} |\ell''(s)|^2\, ds,
\end{multline}
and this allows us to prove this case.
\end{proof}

\subsection{Proof of Theorem \ref{global existence}}
\begin{proof} 
We combine \cref{local-in-time existence} and \cref{aprori est} to establish our result. Note that we can take $\delta_0$ small enough in \cref{local-in-time existence} so that \eqref{initial cond req-2} yields
\begin{equation*}
h_0 \in \Omega^0 \mbox{ and }\rho_0 >0.
\end{equation*}
Since $h_1 \in \Omega^0$, there exists $\eta>0$ such that 
\begin{equation*}
\operatorname{dist}(h_1,\partial\Omega)>1+2\eta.
\end{equation*}
We can assume that $\delta_0\leq \eta$ where $\delta_0$ is the constant in \eqref{initial cond req-2}.

Let us fix 
\begin{equation}\label{choice of delta}
\delta = \min\left(\delta_0,\,\dfrac{\varepsilon_0}{C_*},\, \frac{\varepsilon_0}{C_0\left(1+ \frac{C_*}{\sqrt{k_p(T_*)}}\right)}, \frac{\eta\sqrt{k_p(T_*)}}{C_0}
\right),
\end{equation}
 where the constants $\delta_0$, $C_*$ are appeared in \cref{local-in-time existence}, $\varepsilon_0$, $C_0$ are introduced in \cref{aprori est}. Since $(\rho_0,u_0,h_0,\ell_0,\omega_0)$ satisfies \eqref{initial condition space:global}-\eqref{smallness} and $\delta\leq \delta_0$, we can apply \cref{local-in-time existence} to obtain the existence of solution of system \eqref{continuity}--\eqref{feedback} in $(0,T_*)$ and 
\begin{equation*}
\|(\rho,u,\ell,\omega)\|_{\widehat{\mc{S}}_{0,T_*}}+ \|h_1-h\|_{L^{\infty}(0,T_*)} \leq C_* \left(\|(\rho_0,u_0,\ell_0,\omega_0)\|_{\widehat{\mc{S}}_{0,0}} + |h_1-h_0|\right).
\end{equation*}
In particular, from \eqref{smallness} and \eqref{choice of delta},
\begin{equation}\label{00:23}
\|(\rho,u,\ell,\omega)\|_{\widehat{\mc{S}}_{0,T_*}}+ \|h_1-h\|_{L^{\infty}(0,T_*)} \leq C_*\delta \leq \varepsilon_0 \leq \delta_0.
\end{equation}
Thus $\operatorname{dist}(h(t),\partial\Omega)>1+\eta$ for $t\in [0,T_*]$ and \cref{aprori est} gives
\begin{equation}\label{useapriori}
\|(\rho,u,\ell,\omega)\|_{\widehat{\mc{S}}_{0,T_*}}+ \|\sqrt{k_p}(h_1-h)\|_{L^{\infty}(0,T_*)} \leq C_0 \left(\|(\rho_0,u_0,\ell_0,\omega_0)\|_{\widehat{\mc{S}}_{0,0}} + |h_1-h_0|\right).
\end{equation}
Using that $(\rho,u,h,\ell,\omega)$ is solution of \eqref{continuity}--\eqref{feedback}, one can check that 
$$
(\rho(T_*,\cdot),u(T_*,\cdot),h(T_*),\ell(T_*),\omega(T_*))
$$ 
satisfies the compatibility conditions \eqref{finalcompatibility-1}-\eqref{finalcompatibility-3} and, from \eqref{00:23}, we have
 \begin{equation*}
\|(\rho(T_*,\cdot),u(T_*,\cdot),\ell(T_*),\omega(T_*))\|_{\widehat{\mc{S}}_{T_*,T_*}}
+ |h_1-h(T_*)| 
\leq \delta_0.
\end{equation*}
We can thus apply again \cref{local-in-time existence} to extend our solution on $(T_*,2T_*)$ and using \eqref{useapriori}, we find
\begin{multline}\label{00:56}
\|(\rho,u,\ell,\omega)\|_{\widehat{\mc{S}}_{T_*,2T_*}}+ \|\sqrt{k_p}(h_1-h)\|_{L^{\infty}(T_*,2T_*)}
\\
\leq C_*\left( \|(\rho(T_*,\cdot),u(T_*,\cdot),\ell(T_*),\omega(T_*))\|_{\widehat{\mc{S}}_{T_*,T_*}}+ |h_1-h(T_*)|\right)
\\
\leq \frac{C_*C_0}{\sqrt{k_p(T_*)}}\left(\|(\rho_0,u_0,\ell_0,\omega_0)\|_{\widehat{\mc{S}}_{0,0}} + |h_1-h_0|\right).
\end{multline}
Thus, combining \eqref{useapriori} and \eqref{00:56}, and using \eqref{choice of delta}, we obtain
\begin{multline*}
\|(\rho,u,\ell,\omega)\|_{\widehat{\mc{S}}_{0,2T_*}}+ \|\sqrt{k_p}(h_1-h)\|_{L^{\infty}(0,2T_*)} \leq C_0\left(1+\frac{C_*}{\sqrt{k_p(T_*)}}\right)\left(\|(\rho_0,u_0,\ell_0,\omega_0)\|_{\widehat{\mc{S}}_{0,0}} + |h_1-h_0|\right) \\ \leq C_0\left(1+\frac{C_*}{\sqrt{k_p(T_*)}}\right)\delta \leq \varepsilon_0.
\end{multline*}
Applying \cref{aprori est}, we deduce
\begin{equation}\label{useapriori1}
\|(\rho,u,\ell,\omega)\|_{\widehat{\mc{S}}_{0,2T_*}}+ \|\sqrt{k_p}(h_1-h)\|_{L^{\infty}(0,2T_*)} \leq C_0 \left(\|(\rho_0,u_0,\ell_0,\omega_0)\|_{\widehat{\mc{S}}_{0,0}} + |h_1-h_0|\right).
\end{equation}
In particular
$\operatorname{dist}(h(t),\partial\Omega)>1+\eta$ for $t\in [T_*,2T_*]$.
Moreover, from \eqref{00:56} and \eqref{choice of delta}, 
\begin{equation*}
\|(\rho(2T_*,\cdot),u(2T_*,\cdot),\ell(2T_*),\omega(2T_*))\|_{\widehat{\mc{S}}_{2T_*,2T_*}}
+ |h_1-h(2T_*)| \leq C_0\frac{C_*}{\sqrt{k_p(T_*)}}\delta \leq \varepsilon_0 \leq \delta_0.
\end{equation*}
Then, we repeat the argument on $[jT_*, (j+1)T_*]$, $j\in \mathbb{N}^*$ and we use that $k_p$ is non-decreasing to conclude the proof.
\end{proof}

\section{Proof of Theorem \ref{asymptotic behavior}}\label{sec:Proof of Theorem asymptotic behavior}
This section is devoted to the proof of Theorem \ref{asymptotic behavior}. 
First, from \cref{global existence}, we have 
$$
\rho - \overline{\rho}\in H^{1}(0,\infty; H^2(\mc{F}(t))), \quad
u\in H^1(0,\infty; H^2(\mc{F}(t)), \quad
\ell, \omega\in H^2(0,\infty)
$$
so that (\cite[Corollary 8.9, p.214]{brezis2010functional}), 
\begin{align}
\lim_{t \to \infty}\|\rho(t,.)-\overline{\rho}\|_{H^2(\mc{F}(t))}=0,\quad \lim_{t \to \infty}\|u(t,.)\|_{H^2(\mc{F}(t))}=0, \quad
 \lim_{t \to \infty}\ell(t)=0,\quad \lim_{t \to \infty}\omega(t)=0 \label{solid body vel}.
\end{align}

In the rest of the section, we show $\lim_{t \to \infty} h(t)=h_1$ that completes the proof of Theorem \ref{asymptotic behavior}. 
In order to do this, we need the notion of weak solutions for the problem \eqref{continuity}-\eqref{initial}.
First, we extend $\rho$ and $u$ in $\mathbb{R}^3$ by the formula
$$
{\rho} = \begin{cases}
\rho \mbox{ in }\mc{F}(t), \\ \frac{3m}{4\pi}=\rho_{\mc{B}} \mbox{ in }\mc{B}(t), \\ 0 \mbox{ in }\mathbb{R}^3 \setminus \Omega.
\end{cases}   \quad {u}=\begin{cases}
u \mbox{ in }\mc{F}(t), \\   \ell(t)+\omega(t) \times (x-h(t))=u_{\mc{B}}  \mbox{ in }\mc{B}(t), \\ 0 \mbox{ in }\mathbb{R}^3 \setminus \Omega.
\end{cases}
$$
Then we consider the following notion of weak solutions (see \cite{Feireisl-ARMA}). 
\begin{defin}\label{weaksol:def}
A triplet $({\rho}, {u}, h)$ is a weak solution to \eqref{continuity}-\eqref{initial} on $(0,T)$ if 
\begin{gather*}
{\rho} \geq 0, \quad
{\rho} \in L^{\infty}(0,T; L^{\gamma}(\Omega)) \cap C([0,T];L^1(\Omega)),\quad
 {u}\in L^2(0,T; H^1_{0}(\Omega)),
\\
u=\ell(t)+\omega(t) \times (x-h(t)) \mbox{ in }\mc{B}(t), \quad h'=\ell,
\\
\int\limits_{0}^{T} \int\limits_{\mathbb{R}^3} \left[{\rho}\frac{\partial \phi}{\partial t}+({\rho}{u})\cdot \nabla \phi\right]\, dx\,dt =0, \label{weak density}\\
\int\limits_{0}^{T} \int\limits_{\mathbb{R}^3} \left[b({\rho})\frac{\partial \phi}{\partial t}+(b({\rho}){u})\cdot \nabla \phi+\left(b({\rho})-b'({\rho}){\rho}\right)\operatorname{div}{u}\,\phi\right]\, dx\,dt =0, \label{renormalized}
\end{gather*}
for any $\phi \in C_c^{\infty}((0,T)\times \mathbb{R}^3)$ and for any $b \in C^1(\mathbb{R})$ such that $b'(z)=0$ for $z$ large enough;
\begin{multline}\label{continuity weak}
\int\limits_{0}^{T} \int\limits_{\mathbb{R}^3} \left[({\rho}{u})\cdot \frac{\partial \phi}{\partial t}+ ({\rho}{u} \otimes {u}) : \mathbb{D}(\phi) + a{\rho}^{\gamma} \operatorname{div} \phi\right]\, dx\, dt 
\\
=   \int\limits_{0}^{T} \int\limits_{\mathbb{R}^3}\left(2\mu \mathbb{D}(u) + \lambda \operatorname{div}u\mathbb{I}_3\right) : \mathbb{D}(\phi) \, dx\, dt + \int\limits_{0}^{T} w \cdot \ell_{\phi} \, dt, 
 \end{multline}
for any $\phi \in C_c^{\infty}((0,T)\times \Omega)$, with $\phi(t,y)=\ell_{\phi}(t) + \omega_{\phi}(t)\times (y-h(t))$ in a neighborhood of $\mc{B}(t)$;
for a.e. $t \in [0,T]$, the following energy inequality holds:
\begin{multline*}
\int\limits_{\Omega}\left( \frac{\rho(t,x)}{2}|u(t,x)|^2+\frac{a}{\gamma-1}\rho^{\gamma}(t,x) \right)dx + \int\limits_{0}^{t}\int\limits_{\Omega}\left(2\mu |D({u})|^2+\lambda|\operatorname{div}{u}|^2\right)\,dx\,dt\\\leq C\left(\int\limits_{\left\{{\rho}(0)>0\right\}}\left( \frac{1}{2}\frac{|q(x)|^2}{{\rho}(0,x)} + \frac{a}{\gamma-1}{\rho}^{\gamma}(0,x)\right)\,dx + \int\limits_{0}^{t} w\cdot \ell\ dt \right);
\end{multline*}
and
\begin{equation*}
\rho(0,\cdot)=\rho_0,\quad (\rho u)(0,\cdot)= q, \quad h(0)=h_0. 
\end{equation*}
\end{defin}
We now state a result on the weak compactness of the set of weak solutions to the problem \eqref{continuity}-\eqref{initial}
obtained in \cite[Theorem 9.1]{Feireisl-ARMA}. 
\begin{thm}\label{sequential stability}
Let $(\rho_n, u_n, h_n)$ be a sequence of weak solutions to \eqref{continuity}-\eqref{initial} on $(0,T)\times \Omega$ with the initial condition $(\rho_{0,n},u_{0,n},h_{0,n})$ and forcing term $w_n$ for each $n \geq 1$. Assume that $\{w_n\}$ is a sequence of bounded and measurable functions such that 
\begin{equation*}
w_n \rightarrow w \mbox{ weakly * in }L^{\infty}(0,T),
\end{equation*}
along with
\begin{align}
\rho_{0n}\rightarrow \rho_0 \quad \mbox{in}\quad L^{\gamma}(\mathbb{R}^3),\label{freisel:densityinitial}\\
\rho_{0n}u_{0n}=q_{n}\rightarrow q \quad \mbox{in}\quad L^{1}(\mathbb{R}^3),\label{freisel:momentuminitial}
\end{align}
where $\rho_0, q$ satisfy the following compatibility conditions
\begin{equation}\label{compatibility weak sol:1}
q=0 \mbox{ a.e. on the set }\{x\in \Omega \mid \rho_0=0\},\, \frac{|q|^2}{\rho_0}\in L^1(\Omega).
\end{equation} Moreover, let 
\begin{align}\label{compatibility weak sol:2}
\int\limits_{\left\{\rho_{0n}>0\right\}}\left( \frac{1}{2}\frac{|q_{n}|^2}{\rho_{0n}} + \frac{a}{\gamma-1}\rho_{n}^{\gamma}(0)\right)\,dx \rightarrow \int\limits_{\left\{\rho_{0}>0\right\}}\left( \frac{1}{2}\frac{|q|^2}{\rho_{0}} + \frac{a}{\gamma-1}\rho_{0}^{\gamma}\right)\,dx
\end{align} and
\begin{equation}\label{compatibility weak sol:3}
h_{0n} \rightarrow h_0.
\end{equation}
Then there is a subsequence such that 
\begin{align*}
\rho_{n} \rightarrow \rho \quad \mbox{in}\quad C([0,T];L^1(\mathbb{R}^3)),\\
u_{n} \rightarrow u \quad \mbox{ weakly in}\quad L^2(0,T; H^1_{0}(\Omega)),\\
h_n \rightarrow h \quad \mbox{ uniformly in}\quad  (0,T).
\end{align*}
where $(\rho,u,h)$ is a weak solution of the problem \eqref{continuity}-\eqref{initial} on $(0,T)\times \Omega$ with the initial conditions
$(\rho_0,q,h_0).$
\end{thm}
With the help of above result, we can now prove Theorem \ref{asymptotic behavior}. 
\begin{proof}[Proof of Theorem \ref{asymptotic behavior}]
From \eqref{1451}, there exist $h^* \in \Omega^0$ and $\{t_n\} \subset \mathbb{R}^*_{+}$ such that 
 \begin{equation*} 
t_n \rightarrow \infty,\quad \lim_{n\rightarrow \infty}h(t_n)=h^*.
\end{equation*}
Define
\begin{equation*}
\rho^* = \mathbbm{1}_{\widehat{\mathcal{F}}(h^*)}\overline{\rho} + \mathbbm{1}_{\widehat{\mc{B}}(h^*)}\rho_{\mc{B}}.
\end{equation*}
Writing
\begin{equation*}
\rho (t_n,\cdot) - \rho^* = [\rho(t_n)-\overline{\rho}]\mathbbm{1}_{\mathcal{F}(t_n)} + \overline{\rho}[\mathbbm{1}_{\mathcal{F}(t_n)}-\mathbbm{1}_{\widehat{\mathcal{F}}(h^*)}] + \rho_{\mc{B}}[\mathbbm{1}_{{\mc{B}}(t_n)} - \mathbbm{1}_{\widehat{\mc{B}}(h^*)}],
\end{equation*}
and using \eqref{solid body vel}, we deduce
\begin{align*}
\rho (t_n,\cdot) \xrightarrow{t_n \to \infty} \rho^* \quad \mbox{in}\quad L^{\gamma}(\mathbb{R}^3),\\
\rho(t_n,\cdot) u (t_n,\cdot) \xrightarrow{t_n \to \infty} 0 \quad \mbox{in}\quad L^1(\mathbb{R}^3),\\
\rho(t_n,\cdot) |u (t_n,\cdot)|^2 \xrightarrow{t_n \to \infty} 0 \quad \mbox{in}\quad L^1(\mathbb{R}^3).
\end{align*} 
We set 
\begin{equation*}
\rho_{0n}=\rho(t_n),\quad u_{0n}=u(t_n),\quad h_{0n}=h(t_n),
\end{equation*}
that satisfy \eqref{freisel:densityinitial}, \eqref{freisel:momentuminitial}, \eqref{compatibility weak sol:1}, \eqref{compatibility weak sol:2} 
and \eqref{compatibility weak sol:3} with $\{\rho_{0n}>0\}=\{\rho_0 > 0\} = \Omega$. We also define 
\begin{equation*}
\rho_{n}(t,x)=\rho(t+t_n,x),\quad u_{n}(t,x)=u(t+t_n,x),\quad h_{n}(t)=h(t+t_n),\quad \ell_n(t)=\ell(t+t_n),
\end{equation*}
that is a weak solution to \eqref{continuity}-\eqref{initial} in the sense of \cref{weaksol:def} (since it is a strong solution) with initial conditions $(\rho_{0n}, u_{0n}, h_{0n})$ and with 
\begin{equation*}
w_n(t) = k_p(t) (h_1-h_n(t)) - k_d\ell_n(t).
\end{equation*}
From \cref{global existence}, we have that 
\begin{equation*}
w_n \rightharpoonup \widehat{w} \mbox{ weakly * in }L^{\infty}(0,T).
\end{equation*}
Thus, we can apply \cref{sequential stability} and we deduce that up to a subsequence for $T> 0$:
\begin{align}\label{B}
\begin{split}
&\rho_{n} \rightarrow \widehat{\rho} \quad \mbox{in}\quad C([0,T];L^1(\mathbb{R}^3)),\\
&u_{n} \rightarrow \widehat{u} \quad \mbox{ weakly in}\quad L^2(0,T; H^1_{0}(\Omega)),\\
&h_n \rightarrow \widehat{h} \quad  \mbox{in}\quad  L^{\infty}(0,T),
\end{split}
\end{align}
with $(\widehat{\rho},\widehat{u},\widehat{h})$ is a weak solution of \eqref{continuity}-\eqref{initial} such that 
\begin{equation*}
\widehat{\rho}(0,\cdot)=\rho^*, \quad (\widehat{\rho}\widehat{u})(0,\cdot)=0,\quad \widehat{h}(0)=h^*,
\end{equation*}
and with 
\begin{equation*}
\widehat{w}(t) = k_p(t)(h_1-\widehat{h}(t)) - k_d\widehat{\ell}(t).
\end{equation*}
Moreover up to a subsequence,
\begin{equation*}
\int\limits_{0}^{T} \|\mathbb{D}(u_n(t,\cdot))\|_{L^2(\Omega)}^2\, dt = \int\limits_{t_n}^{t_n + T}\|\mathbb{D}(u(t,\cdot))\|_{L^2(\Omega)}^2\, dt\xrightarrow {n \to \infty} 0.
 \end{equation*}
 The above limit and \eqref{B} yield
 \begin{equation*}
 \mathbb{D}\widehat{u} =0 \mbox{ in }(0,T)\times \Omega.
 \end{equation*}
 Thus, we deduce that $\widehat{u}=0$ in $(0,T)\times \Omega$. In particular, we have $\widehat{h}'(t)=0$, $\forall\, t\in(0,T)$. This gives, 
 \begin{equation*}
 \widehat{h}=h^* \mbox{ in }(0,T).
 \end{equation*}
 Consequently, \eqref{continuity weak} gives
 \begin{equation*}
 \int\limits_0^T\int\limits_{\mathbb{R}^3} a (\widehat{\rho})^{\gamma} \operatorname{div} \phi\ dx\ dt = \int\limits_0^T k_p (h_1-h^*)\cdot \ell_{\phi}\ dt,
 \end{equation*}
 for all $\phi\in C_c^{\infty}((0,T)\times \Omega),$ with 
 $\phi(t,y)=\ell_{\phi}(t) + \omega_{\phi}(t)\times (y-h(t))$ in a neighborhood of $\mc{B}(t)$.
 Then we take 
 \begin{equation*}
\operatorname{div}\phi =0,\ \phi(t,\cdot)=(h_1-h^*)\zeta(t) \mbox{ in }\mc{B}(t), \mbox{ with }\zeta \in C_c^{\infty}((0,T)),
 \end{equation*}
 so that
 \begin{equation*}
 \int\limits_0^T |h_1-h^*|^2 k_p(t) \zeta(t)\ dt=0, \quad \forall \ \zeta \in C^{\infty}_c((0,T)).
 \end{equation*}
 Since, $k_p \neq 0$, $h^*=h_1$.
\end{proof}

\medskip
\bibliographystyle{plain}
\bibliography{references}
\end{document}